\documentclass[9pt,reqno]{amsart}

\usepackage{amsmath, amsfonts, amssymb, latexsym, amsthm}
\usepackage[pagewise]{lineno}
\usepackage{amsmath, amsfonts, amssymb, latexsym}
\usepackage[numbers,sort&compress]{natbib}
\usepackage{mathrsfs}
\usepackage{pdfcomment}
\usepackage{xcolor}
\usepackage{esint}

\hypersetup{hidelinks,
	colorlinks=true,
	allcolors=black,
	pdfstartview=Fit,
	breaklinks=true
}

\everymath{\displaystyle}

\numberwithin{equation}{section}

\newcommand{\cU}{\mathcal{U}}

\newcommand{\R}{\mathbb{R}}

\DeclareMathOperator{\essinf}{ess\,inf}

\def\Om{\Omega}

\newtheorem{lemma}{Lemma}[section]
\newtheorem{theorem}[lemma]{Theorem}
\newtheorem{remark}[lemma]{Remark}

\newtheorem{proposition}[lemma]{Proposition}

\theoremstyle{definition}
\newtheorem{definition}[lemma]{Definition}

\makeatletter

\@addtoreset{equation}{section}
\makeatother

\begin{document}
	\title{The solvability and a Stackelberg-Nash game problem for degenerate elliptic equations}
	
	
	\author{\sffamily Yuanhang Liu$^{1}$, Weijia Wu$^{1,*}$, Donghui Yang$^1$, Xu Zhang$^1$   \\
		{\sffamily\small $^1$ School of Mathematics and Statistics, Central South University, Changsha 410083, China. }\\
	}
	\footnotetext[1]{Corresponding author: weijiawu@yeah.net }
	
	\email{liuyuanhang97@163.com}
	\email{weijiawu@yeah.net}
	\email{donghyang@outlook.com}
	\email{xuzhang174@163.com}
	
	\keywords{}
	\subjclass[2020]{35J70,35J25}
	
	\maketitle
	
	\begin{abstract}
		This study aims to investigate the functional properties of weak solution spaces and their compact embedding properties in relation to the Dirichlet problem associated with a specific class of degenerate elliptic equations. To expand the scope of analysis for degenerate elliptic problems, we employ weighted Sobolev inequalities and compact embedding techniques within the framework of weighted Sobolev spaces. Additionally, we apply these findings to examine a Stackelberg-Nash game problem and obtain the existence of the  Stackelberg-Nash equilibrium.

	\end{abstract}
	
	\pagestyle{myheadings}
	\thispagestyle{plain}
	\markboth{DEGENERATE ELLIPTIC EQUATIONS AND STACKELBERG-GAME PROBLEM}{YUANHANG LIU, WEIJIA WU, DONGHUI YANG AND XU ZHANG}

	\section{Introduction}
	In a wide range of disciplines, including modern physics, mathematics, probability theory, and control theory, numerous linear and nonlinear problems can be ultimately formulated as solving various types of partial differential equations. As a result, researchers from diverse fields have been drawn to investigate the solvability of such equations, establishing it as a prominent area of research. Degenerate elliptic operators, which extend the classical elliptic operators, hold significant importance in areas like stochastic processes, quantum mechanics, and differential geometry. Consequently, the study of solvability for the Dirichlet problem associated with degenerate elliptic equations has also garnered attention, albeit with relatively limited findings thus far. In \cite{cavalheiro2002approximation}, the author explored a significant finding concerning the approximation of weak solutions for degenerate elliptic equations. Specifically, it was established that a sequence of solutions derived from non-degenerate elliptic equations can serve as a reliable approximation for the weak solution of degenerate elliptic equations. Moreover, in \cite{cavalheiro2006solvability}, the author successfully demonstrated the existence and uniqueness of entropy solutions for the Dirichlet problem associated with a specific class of degenerate elliptic equations. Importantly, these results were achieved through the application of established methodologies outlined in \cite{benilan19951}. In \cite{cavalheiro2008regularity}, the author established the existence of higher order weak derivatives of weak solutions of the Dirichlet problem for a class of degenerate elliptic equations. For other degenerate elliptic equations, see \cite{monticelli2015existence,di2020existence,farina2013monotonicity,boimatov2002estimates}.
	
	This study focuses on exploring the functional properties of weak solution spaces and their compact embedding properties for the Dirichlet problem associated with a specific class of degenerate elliptic equations. To extend the analysis of degenerate elliptic problems, we utilize weighted Sobolev inequalities and compact embedding techniques within the framework of weighted Sobolev spaces. Moreover, we provide an application of these findings by investigating a game problem. Differential games were introduced originally by Isaacs (see
	\cite{isaacs1965}). Since then,
	lots of researchers were attracted to establish and improve the related theory. Meanwhile,
	the theory was applied to a large number of fields. For a comprehensive survey on the
	differential game theory, we refer to
	\cite{bacsar1998dynamic,elliott1972existence,evans1984differential,friedman1971} and the references therein. The famous game problem, Stackelberg-Nash game problem, is an interesting question in the framework of mathematical finance. For instance, it is well
	known that the price of an European call option is governed by a backward PDE. The independent space variable must be interpreted as the stock price and the time variable is in fact the reverse of time. In this regard, it can be interesting to control the solution of the system with the composed action
	of several agents, each of them corresponding to a different range of values of the space variable. For further information on
	the modeling and control of phenomena of this kind, see for instance \cite{ross2011elementary,wilmott1995mathematics}. For more Stackelberg–Nash game problems, we
	refer the reader to the works in \cite{diaz2004approximate,guillen2013approximate}.
	
	The following sections of this manuscript are structured as follows: Section 2 presents a comprehensive overview of the key research questions and highlights the significant findings. In Section 3, we provide crucial supplementary findings that are essential to the analysis. Section 4 establishes the well-posedness of the degenerate elliptic equations. Finally, Section 5 investigates the Stackelberg-Nash game problems.

	\section{Problems formulation and main results}
	In this section, we will outline the principal issues investigated and present the pivotal
	findings of this manuscript.  Firstly, let us introduce necessary notations.
	
	Let $\alpha\in ( 0,1] $, $\Omega=(0,1)\times(0,1)$. $\omega \subset \Omega$ is a nonempty subset and $\chi_{\omega}$ is the corresponding characteristic function. We denote by $|\cdot|$ the Lebesgue measure on $\R$.
	
	The system, that we consider in this paper, is described by the following degenerate elliptic
	equation:
	\begin{equation}\label{1.1}
		\begin{cases}
			\frac{1}{2}\partial_{xx}u-x^{\alpha}\partial_{y}u=f, & (x,y)\in \Om,\\
			u(x,y)=0, & (x,y)\in \partial \Om,
		\end{cases}
	\end{equation}
	where $u$ is the state variable, $f\in L^2(\Omega;x^{-\alpha})$.
	
	Denote
	\begin{equation*}
		\begin{split}
			W^{1,1}(\Omega;x^{\alpha})
			&:=\{u\in L^{2}(\Omega): \partial_{x} u  \in L^{2}(\Omega), x^{\frac{\alpha}{2}}\partial_{y} u  \in L^{2}(\Omega) \},\\
			V(\Omega)
			&:=\{u\in W^{1,1}(\Omega;x^{\alpha}): \partial_{x}\partial_{y} u  \in L^{2}(\Omega)\},\\
			L^2(\Omega;x^{\alpha})
			&:=\{f\in L^{2}(\Omega): x^{-\alpha}f  \in L^{2}(\Omega) \},
		\end{split}
	\end{equation*}
	endowed with the norms
	\begin{equation*}
		\begin{split}
			\|u\|_{W^{1,1}(\Omega;x^{\alpha})}&=(\|u\|_{L^{2}(\Omega)}^2+\|\partial_x u\|_{L^{2}(\Omega)}^2 +\|x^{\frac{\alpha}{2}}\partial_{y} u\|_{L^{2}(\Omega)}^2  )^{\frac{1}{2}},\\
			\|u\|_{V(\Omega)}&=(\|u\|_{W^{1,1}(\Omega;x^{\alpha})}^2+\|\partial_{x}\partial_{y} u\|_{L^{2}(\Omega)}^2   )^{\frac{1}{2}},\\
			\|f\|_{L^2(\Omega;x^{-\alpha})}&=\|x^{-\alpha}f\|_{L^2(\Omega)},
		\end{split}
	\end{equation*}
	respectively.
	
	Denote by $W_{0}^{1,1}(\Omega;x^\alpha)$ and $V_{0}(\Omega)$, the closure of $C_0^\infty (\Om)$ in the spaces $W^{1,1}(\Omega;x^\alpha)$ and $V(\Omega)$, respectively. i.e.,
	\begin{equation*}
		\begin{split}
			W_{0}^{1,1}(\Omega;x^\alpha)&=\overline{C_0^\infty (\Om)}^{W^{1,1}(\Omega;x^{\alpha})},\\
			V_{0}(\Omega)&=\overline{C_0^\infty (\Om)}^{V(\Omega)}.
		\end{split}
	\end{equation*}
	
	Denote $\mathcal{D}'(\Omega)=\left(C_0^\infty(\Omega)\right)'$.

	\begin{definition}\label{3.8}
		We say that $u\in W_0^{1,1}(\Omega;x^\alpha)$ is a weak solution of \eqref{1.1}, if for any $\varphi\in C_0^\infty(\Omega)$, 
		\begin{equation}\label{3.2.9}
			\iint_{\Omega} \left[(x^\alpha\partial_yu)\varphi+\frac{1}{2}(\partial_xu)( \partial_x\varphi)\right]dxdy=\iint_\Omega f\varphi dxdy.
		\end{equation}
	\end{definition}

	The main result of this paper is the following well-posedness for the degenerate elliptic equation \eqref{1.1}:
	\begin{theorem}\label{3.12}
		Let $f\in L^2(\Omega,x^{-\alpha})$. Then the equation \eqref{1.1} has a unique weak solution $u\in W_{0}^{1,1}(\Omega;x^\alpha)$. Moreover, there is a constant $C>0$ such that the following inequality holds:
		\begin{equation}
			\label{ine44}
			\|u\|_{W^{1,1}(\Omega;x^\alpha)}\leq C\|f\|_{L^2(\Omega,x^{-\alpha})}.
		\end{equation}
		
	\end{theorem}
	
	As an important application of degenerate elliptic equation \eqref{1.1}, we consider the following degenerate problem with three controls (one leader and two followers), but very similar considerations hold for other systems (degenerated cases or non-degenerated cases) with a
	higher number of controls.
	\begin{equation}\label{1.1-1}
		\begin{cases}
			\frac{1}{2}\partial_{xx}u-x^{\alpha}\partial_{y}u=\chi_\omega g+\chi_{\omega_1} f_1+\chi_{\omega_2} f_2,, & (x,y)\in \Om,\\
			u(x,y)=0, & (x,y)\in \partial \Om,
		\end{cases}
	\end{equation}
	where $\omega,\omega_1,\omega_2$ are subsets of $\Omega$. $\omega$ is the main (leader) control domain and $\omega_1,\omega_2$ are the secondary (followers) control domains. The controls $g,f_1,f_2\in L^2(\Omega,x^{-\alpha})$, where $g$ is the main (leader) control and $f_1$ and $f_2$ are the secondary (followers) controls. The solution of system (\ref{1.1-1}) denote by $y(g,f_1,f_2)$.
	
	We define the following admissible sets of controls:
	\begin{equation*}
		\cU_1:=\big\{f_1\in L^2(\Omega;x^{-\alpha}):\|f_1\|_{L^2(\Omega;x^{-\alpha})}\leq M_1  \big\},
	\end{equation*}
	\begin{equation*}
		\cU_2:=\big\{f_2\in L^2(\Omega;x^{-\alpha}):\|f_2\|_{L^2(\Omega;x^{-\alpha})}\leq M_2  \big\}.
	\end{equation*}
	where $M_1,M_2$ are positive constants. Meanwhile, we introduce the following two functionals: For each $i=1,2$,
	the functional $J_i: (L^2(\Omega;x^{-\alpha}))^2\rightarrow [0,+\infty)$ is defined by
	\begin{equation}\label{Intro-2}
		J_i(f_1,f_2):= \|y(g,f_1,f_2)-y^i_d\|^2_{L^2(G_i)}+\|f_i\|^2_{L^2(\omega_i;x^{-\alpha})},
	\end{equation}
	where $G_i\subset\Omega\ (i=1,2)$  representing observation domains for the followers and  $y^i_d\in L^2(\Omega)$ and
	\begin{equation}\label{y1-y2}
		y^1_d\not=y^2_d.
	\end{equation}
	
	The mathematical model for the control problem is:
	
	The followers $f_1$ and $f_2$ assume that the leader $g $ has made a choice and intend to be a Nash equilibrium for the costs functionals $J_i(f_1,f_2),i=1,2$. Thus, once the leader $g$ has been fixed,
	
	{\bf(P)}\;\;Does there exist $(f_1^*,f_2^*)\in {\mathcal U}_1\times {\mathcal U}_2$ with respect to $g$ so that
	\begin{equation}\label{20180208-1}
		J_1(f_1^*,f_2^*)\leq J_1(f_1,f_2^*)\;\;\mbox{for all}\;\;f_1\in {\mathcal U}_1
	\end{equation}
	and
	\begin{equation}\label{20180208-2}
		J_2(f_1^*,f_2^*)\leq J_2(f_1^*,f_2)\;\;\mbox{for all}\;\;f_2\in {\mathcal U}_2?
	\end{equation}
	We call the problem {\bf(P)} as a {\it Stackelberg-Nash game problem}. It is a noncooperative
	game problem of the followers. If the answer to the problem {\bf(P)} is yes, we call $(f_1^*,f_2^*)$ a {\it Stackelberg-Nash equilibrium} (or an {\it optimal
		strategy pair}, or an {\it optimal control pair}) of {\bf(P)}. We can understand the problem {\bf(P)} in the following manner:
	There are two followers executing their strategies and hoping to achieve their goals $y_d^1$ and $y_d^2$, respectively.
	If the first follower chooses the strategy $f_1^*$, then the second follower can execute the strategy $f_2^*$
	so that $y(g,f_1^*,f_2^*)$ is closer to $y_d^2$; Conversely, if the second follower chooses the
	strategy $f_2^*$, then the first follower can execute the strategy $f_1^*$ so that $y(g,f_1^*,f_2^*)$ is closer
	to $y_d^1$. Roughly speaking, if one follower is deviating from $(f_1^*,f_2^*)$, then the cost functional of this
	follower would get larger; and there is no information given if both followers are deviating from
	the Nash equilibrium $(f_1^*,f_2^*)$. From the knowledge of game theory, in generally,
	Nash equilibrium is not unique.
	
	The main result of the
	{\it Stackelberg-Nash game problem}  {\bf(P)} is about the existence of a Stackelberg-Nash equilibrium.
	\begin{theorem}
		\label{Intro-3}
		Let $g\in L^2(\Omega;x^{-\alpha})$ be given. The problem {\bf(P)} with respect to the system (\ref{1.1}) admits a  Stackelberg-Nash equilibrium, i.e., there exist
		at least a pair of $(f_1^*,f_2^*)\in {\mathcal{U}}_1\times {\mathcal{U}_2}$
		with respect to $g$ so that (\ref{20180208-1}) and (\ref{20180208-2}) hold.
	\end{theorem}
	
	Several remarks are given in order as follows.
	\begin{remark}
		Since $C_0^\infty(\Omega)$ is dense in  $ W_0^{1,1}(\Omega;x^\alpha)$, the test function $\varphi$ in the integral equation \eqref{3.2.9} can be taken as an arbitrary function in $W_0^{1,1}(\Omega;x^\alpha)$ .
	\end{remark}
	
	\begin{remark}
		The assumption of $\alpha\in (0,1]$ is crucial in this context. This choice allows us to establish the compactness property $W_{0}^{1,1}(\Omega;x^\alpha) \hookrightarrow L^2(\Omega)$, as demonstrated in Proposition \ref{compactembed} (to be presented later). However, if $\alpha\notin (0,1]$, we have not yet obtained a similar result. It is noteworthy that in a particular case when $\alpha=1$, equation \eqref{1.1} corresponds to a Kolmogorov-type equation. Our findings illustrate the extension of our results to classical Kolmogorov-type equation, thereby showcasing the broader applicability and significance of our research. 
		For more Kolmogorov-type equation, such as parabolic case, we
		refer the reader to the work in \cite{beauchard2009some,le2016null,huang2019lp}.
	\end{remark}
	\begin{remark}
		We focus our attention on the function spaces $V(\Omega)$ and $L^2(\Omega,x^{-\alpha})$ for some reasons. On one hand, the existence of the solution cannot be directly established using the classical Lax-Milgram theorem. Instead, we need to demonstrate the validity of a generalized Lax-Milgram theorem (refer to Lemma \ref{3.12L}). On the other hand, the choice of $L^2(\Omega,x^{-\alpha})$ in this context may initially appear unconventional. However, it is motivated by technical constraints outlined in Theorem \ref{3.12}. This observation highlights the inherent distinctions between degenerate elliptic equations and their non-degenerate counterparts.
	\end{remark}
	\begin{remark}
		The presence of certain captivating inquiries is evident in the following: What nature of problem emerges when the classical Kolmogorov operator lacks symmetry and eigenvalues? Does the degenerate elliptic operator associated with the equation under investigation exhibit eigenvalues?
	\end{remark}
	\begin{remark}
		In prior research on control problems for degenerate elliptic equations, numerous scholars have employed approximation methodologies to examine this matter, see \cite{barles1995dirichlet,lou2003maximum,barles1990comparison}. Nevertheless, in the context of the control problem explored in this manuscript, the existence of Stackelberg-Nash equilibrium is deduced through the utilization of pertinent solution space properties. This approach introduces a fresh perspective for investigating control problems associated with degenerate elliptic equations.
	\end{remark}
	
	\section{Auxiliary conclusions}
	In this section, we present several supplementary findings that will be utilized in subsequent analyses. In what following, we denote the norms $\|\cdot\|_{W^{1,1}(\Omega;x^{\alpha})}$ as $\|\cdot\|_{W^{1,1}}$ and $\|\cdot\|_{L^2(\Omega;x^{\alpha})}$ as $\|\cdot\|_{L^2}$ for simplifying the notations.
	\begin{lemma}\label{3.1}
		Let $u\in \mathcal{D}'(\Omega)$, then we have
		\begin{equation*}
			x^\frac{\alpha}{2}\partial_yu\in L^2(\Omega)\Leftrightarrow\partial_y(x^\frac{\alpha}{2}u)\in L^2(\Omega).
		\end{equation*}
	\end{lemma}
	\begin{proof}
		For any $\varphi\in C_0^\infty(\Omega)$, we have $x^\frac{\alpha}{2}\varphi\in C_0^\infty(\Omega)$, Thus, from the definition of weak derivative we have
		\begin{equation*}
			\begin{split}
				\iint_\Omega (x^\frac{\alpha}{2}\partial_yu)\varphi dxdy
				&=\iint_\Omega(\partial_yu)(x^\frac{\alpha}{2}\varphi)dxdy=-\iint_\Omega u(\partial_y(x^\frac{\alpha}{2}\varphi))dxdy\\
				&=-\iint_\Omega(x^\frac{\alpha}{2}u)\partial_y\varphi dxdy=\iint_\Omega \varphi (\partial_y(x^\frac{\alpha}{2}u))dxdy.
			\end{split}
		\end{equation*}
		Hence $\forall u \in \mathcal{D}'(\Omega)$,   $x^\frac{\alpha}{2}\partial_yu=\partial_y(x^\frac{\alpha}{2}u)$.  i.e. $x^\frac{\alpha}{2}\partial_yu\in L^2(\Omega)\Leftrightarrow\partial_y(x^\frac{\alpha}{2}u)\in L^2(\Omega)$.
	\end{proof}
	
	\begin{proposition}
		\label{3.2}
		$\left(W^{1,1}(\Omega;x^\alpha), (\cdot,\cdot)_{W^{1,1}}\right)$ is a  Hilbert space.
	\end{proposition}
	\begin{proof}
		On one hand, it is easy to verify that $(W^{1,1}(\Omega;x^{\alpha}),(\cdot,\cdot)_{W^{1,1}})$ is an inner product space.
		On the other hand, if $\{u_n\}_{n\in\mathbb{N}}\subset W^{1,1}(\Omega;x^\alpha)$ is a  Cauchy sequence, then  $\{u_n\}_{n\in\mathbb{N}},  \{\partial_xu_n\}_{n\in\mathbb{N}}, \{x^\frac{\alpha}{2}\partial_yu_n\}_{n\in\mathbb{N}}$ is a Cauchy sequence of $L^2(\Omega)$. Therefore, there exist  $u_0, v_0, w_0\in L^2(\Omega)$ such that when  $n\rightarrow\infty$ , we have 
		\begin{equation}\label{proof3.1.1}
			u_n\stackrel{\mathrm s}{\to} u_0,\ \partial_xu_n\stackrel{\mathrm s}{\to} v_0,\ \partial_y(x^\frac{\alpha}{2}u_n)=x^\frac{\alpha}{2}\partial_yu_n\stackrel{\mathrm s}{\to} w_0
		\end{equation}
		in the space $ L^2(\Omega)$.
		For any  $\varphi\in C_0^\infty(\Omega)$,
		\begin{equation*}
			\begin{split}
				\iint_\Omega (\partial_xu_n)\varphi dxdy
				&=-\iint_\Omega u_n\partial_x\varphi dxdy\rightarrow -\iint_\Omega u_0\partial_x\varphi dxdy,\\
				\iint_\Omega (x^\frac{\alpha}{2}\partial_yu_n)\varphi dxdy
				&=\iint_\Omega (\partial_y(x^\frac{\alpha}{2}u_n))\varphi dxdy\\
				&=-\iint_\Omega x^\frac{\alpha}{2}u_n \partial_y\varphi dxdy\rightarrow -\iint_\Omega x^\frac{\alpha}{2}u_0\partial_y\varphi dxdy.
			\end{split}
		\end{equation*}
		By Lemma \ref{3.1} and \eqref{proof3.1.1}, it is easy to see that in $\mathcal{D}'(\Omega)$ it holds that $v_0=\partial_xu_0, w_0=\partial_y(x^\frac{\alpha}{2}u_0)=x^\frac{1} {2}\partial_yu_0$. Thus  we have $v_0=\partial_xu_0, w_0=\partial_y(x^\frac{\alpha}{2}u_0)=x^\frac{\alpha}{2}\partial_yu_0$ in $L^2(\Omega)$, so $u_0\in W^{1,1}(\Omega;x^{\alpha})$. Combining with \eqref{proof3.1.1} we have 
		\begin{equation*}	
			u_n\stackrel {\mathrm{s}}{\to}u_0.
		\end{equation*}	
		in $W^{1,1}(\Omega;x^\alpha)$. The proof is completed.
	\end{proof}

	\begin{proposition}\label{3.3}
		The following statements are true:
		\begin{itemize}
			\item [(1)] $H^1(\Omega)\subsetneq W^{1,1}(\Omega;x^\alpha)$;
			\item [(2)] $H_0^1(\Omega)\subset W_0^{1,1}(\Omega;x^\alpha)$.
		\end{itemize} 
	\end{proposition}
	\begin{proof}
		We divide the proof into two steps.
		
		{\it Step 1. Show the conclusion (1).} 
		
		Let $u\in H^1(\Omega)$, then $u, \partial_x u, \partial_y u\in L^2(\Omega)$, and
		\begin{equation}\label{proof3.1.2}	
			\begin{split}
				\|u\|_{W^{1,1}(\Omega;x^{\alpha})}
				&=\left(\|u\|_{L^2(\Omega)}^2
				+\|\partial_xu\|_{L^2(\Omega)}^2+\|x^{\frac{\alpha}{2}}\partial_yu\|_{L^2(\Omega)}^2\right)^{\frac{1}{2}}\\
				&\leq C\left(\|u\|_{L^2(\Omega)}^2
				+\|\partial_xu\|_{L^2(\Omega)}^2+\|\partial_yu\|_{L^2(\Omega)}^2\right)^{\frac{1}{2}}\\
				&=C\|u\|_{H^1(\Omega)}.
			\end{split}
		\end{equation}
		Hence, $u\in W^{1,1}(\Omega;x^{\alpha})$.
		
		Next let $u(x,y)=\left(x^2+y\right)^\frac{1}{4}, (x,y)\in \Omega$, $\alpha=\frac{1}{2}$, then we have
		\begin{equation*}
			\begin{split}
				\partial_xu=\frac{x}{2}\left(x^2+y\right)^{-\frac{3}{4}},\quad \partial_yu=\frac{1}{4}\left(x^2+y\right)^{-\frac{3}{4}},
			\end{split}
		\end{equation*}
		and
		\begin{equation*}
			\begin{split}
				\|u\|_{L^2}^2
				&=\iint_\Omega (x^2+y)^\frac{1}{2} dxdy\leq \sqrt{2},\\
				\|\partial_xu\|_{L^2}^2
				&=\frac{1}{4}\iint_\Omega \frac{x^2}{\left(x^2+y\right)^\frac{3}{2}} dxdy\leq \frac{1}{4}\iint_\Omega \frac{1}{\sqrt{x^2+y}}dxdy\\
				&=\frac{1}{4}\int_0^1\int_{x^2}^{x^2+1}\frac{1}{\sqrt{z}} dzdx
				=\frac{1}{4}\int_0^1 2\sqrt{z}\Big|_{x^2}^{x^2+1} dx\leq \sqrt{2},\\
				\|x^\frac{\alpha}{2}\partial_yu\|_{L^2}^2
				&=\frac{1}{16}\iint_\Omega \frac{x^\frac{1}{2}}{\left(x^2+y\right)^\frac{3}{2}}dxdy\leq\frac{1}{16}\iint_\Omega \frac{(x^2+y)^\frac{1}{4}}{\left(x^2+y\right)^\frac{3}{2}} dxdy\\
				&=\frac{1}{16}\iint_\Omega \frac{1}{\left(x^2+y\right)^\frac{5}{4}} dxdy=\frac{1}{16}\int_0^1 dx\int_{x^2}^{x^2+1}\frac{1}{z^\frac{5}{4}}dz\\
				&=-\frac{1}{4}\int_0^1z^{-\frac{1}{4}}\Big|_{x^2}^{x^2+1} dx<\infty.
			\end{split}
		\end{equation*}
		Thus we have $u\in W^{1,1}(\Omega;x^\alpha)$, but
		\begin{equation*}
			\begin{split}
				\|\partial_yu\|_{L^2}^2
				&=\frac{1}{16}\iint_\Omega \frac{1}{\left(x^2+y\right)^\frac{3}{2}} dxdy
				=\frac{1}{16}\int_0^1 dx\int_0^1\frac{1}{\left(x^2+y\right)^\frac{3}{2}} d(x^2+y)\\
				&=\frac{1}{16}\int_0^1 dx\int_{x^2}^{x^2+1}\frac{1}{z^\frac{3}{2}} dz=-\frac{1}{8}\int_0^1z^{-\frac{1}{2}}\Big|_{x^2}^{x^2+1} dx\\
				&=-\frac{1}{8}\int_0^1\left(\frac{1}{\sqrt{x^2+1}}-\frac{1}{x}\right)dx=+\infty.
			\end{split}
		\end{equation*}
		Therefore, $u\notin H^1(\Omega)$.
		
		{\it Step 2. Show the conclusion (2).} 
		
		For any $u\in H_0^1(\Omega)$, there exists a sequence  $\{u_n\}_{n\in\mathbb{N}}\subset C_0^\infty(\Omega)$ such that
		\begin{equation*}
			\|u_n-u\|_{H^1}\rightarrow  0, n\rightarrow \infty.
		\end{equation*}
		By \eqref{proof3.1.2} we can deduce that
		\begin{equation*}
			\|u_n-u\|_{W^{1,1}}\rightarrow 0, n\rightarrow \infty.
		\end{equation*}
		Thus $u\in W_0^{1,1}(\Omega;x^\alpha)$. The proof is completed.
		
	\end{proof}
	
	\begin{lemma}\label{3.4}
		Let $u\in W^{1,1}(\Omega;x^\alpha)$, then for any $\emptyset \neq V\subset \subset \Omega$, there exists a sequence $\{u_n\}_{n\in\mathbb{N}}\subset C^\infty(\overline{V}), u\in W^{1,1}(V;x^\alpha)$, such that
		\begin{equation*}
			u_n\rightarrow u ~\text{strongly in}~W^{1,1}(V;x^\alpha),
		\end{equation*}
		where $V\subset\subset \Omega$,  $\overline{V}\subset U$ is a compact set, and
		\begin{equation*}
			W^{1,1}(V;x^\alpha)=\left\{u\in V\mid u\in W^{1,1}(\Omega;x^\alpha)\right\}.
		\end{equation*}		
	\end{lemma}
	
	\begin{proof}
		Since $V\subset \subset \Omega$, there exists $\epsilon>0$ such that
		\begin{equation*}
			V\subset [\epsilon, 1-\epsilon]\times [\epsilon, 1-\epsilon]=	U.
		\end{equation*}	
		If $u\in W^{1,1}(\Omega;x^\alpha)$, then
		\begin{equation*}
			\begin{split}
				\|u\|_{H^1(U)}^2
				&=\|u\|_{L^2}^2+\|\partial_xu\|_{L^2}^2+\|\partial_yu\|_{L^2}^2\\
				&\leq \|u\|_{L^2}^2+\|\partial_xu\|_{L^2}^2+\frac{1}{\epsilon^\frac{\alpha}{2}}\|x^\frac{\alpha}{2}\partial_yu\|_{L^2}^2<+\infty.
			\end{split}
		\end{equation*}
		Thus $u\in H^1([\epsilon, 1-\epsilon]\times [\epsilon, 1-\epsilon])$. By the classical Sobolev space local approximation theorem, there exists a sequence  $\{u_n\}_{n\in\mathbb{N}}\subset C^\infty(\overline{V})$, such that 
		\begin{equation*}
			u_n\rightarrow u, \  ~\text{strongly in}~ \ H^1([\epsilon, 1-\epsilon]\times [\epsilon, 1-\epsilon]).
		\end{equation*}	
		Thus, we have
		\begin{equation*}
			u_n\rightarrow u  \  ~\text{strongly in}~ \  H^1(V).
		\end{equation*}	
		Hence
		\begin{equation*}
			\begin{split}
				\|u_n-u\|_{W^{1,1}(V;x^\alpha)}
				&=\left(\|u_n-u\|_{L^2(V)}^2+\|\partial_xu_n-\partial_xu\|_{L^2(V)}^2+\|x^\frac{\alpha}{2}\partial_yu_n-x^\frac{\alpha}{2}\partial_yu\|_{L^2(V)}^2\right)^{\frac{1}{2}}\\
				&\leq \|u_n-u\|_{H^1(V)}\rightarrow 0,  ~\text{as}~n\rightarrow\infty.
			\end{split}
		\end{equation*}	
		Thus we have completed the proof.	
	\end{proof}

	\begin{proposition}\label{3.5}
		Let $u\in W_0^{1,1}(\Omega;x^\alpha)$, then we have $u|_{\partial\Omega}=0$ in $L^2(\partial\Omega)$.	
	\end{proposition}
	
	\begin{proof}
		Let  $u\in W_0^{1,1}(\Omega;x^\alpha)$, then there exists a sequence $\{u_n\}_{n\in\mathbb{N}}\subset C_0^\infty(\Omega)$ such that
		\begin{equation}\label{proof3.1.3}
			\|u_n-u\|_{W^{1,1}}\rightarrow 0, ~\text{as}~ \ n\rightarrow\infty,
		\end{equation}	
		which implies that  $\{u_n\}_{n\in\mathbb{N}}\subset C_0^\infty(\Omega)$ is a Cauchy sequence in $W_0^{1,1}(\Omega;x^\alpha)$. Thus we have
		\begin{equation}\label{proof3.1.4}
			\begin{split}
				\|u_n-u_m\|_{W^{1,1}}^2=&\|u_n-u_m\|_{L^2}^2+\|\partial_xu_n-\partial_xu_m\|_{L^2}^2\\
				&+\|x^\frac{\alpha}{2}\partial_yu_n-x^\frac{\alpha}{2}\partial_yu_m\|_{L^2}^2\rightarrow 0, ~\text{as}~n,m\rightarrow\infty.
			\end{split}
		\end{equation}
		Since  $\{u_n\}_{n\in\mathbb{N}}\subset C_0^\infty(\Omega)$, for any $n\in\mathbb{N}$ and $x, y\in (0,1)$, we can deduce 
		\begin{equation}\label{proof3.1.5}
			\begin{split}
				u_n(x,y)
				&=u_n(x,y)-u_n(0,y)=\int_0^x\partial_xu_n(s,y)ds,\\
				x^\frac{\alpha}{2}u_n(x,y)
				&=x^\frac{\alpha}{2}u_n(x,y)-x^\frac{\alpha}{2}u_n(x,0)=\int_0^y\partial_y\left(x^\frac{\alpha}{2}u_n(x,s)\right)ds\\ &=\int_0^yx^\frac{\alpha}{2}\partial_yu_n(x,s)ds.
			\end{split}
		\end{equation}
		
		Next, the rest of the proof will be carried out by
		the following steps.
		
		{\it Step 1.} On one hand,
		by \eqref{proof3.1.5}, we have
		\begin{equation*}
			\int_0^1|u_n(x,y)|^2 dy=\int_0^1\left|\int_0^x\partial_xu_n(s,y) ds\right|^2 dy\leq \iint_\Omega |\partial_xu_n(x,y)|^2 dxdy.
		\end{equation*}
		According to \eqref{proof3.1.4}, we hold
		\begin{equation*}
			\begin{split}
				\|u_n(x,\cdot)-u_m(x,\cdot)\|^2_{L^2(0,1)}&\leq \|\partial_x u_n-\partial_xu_m\|^2_{L^2(\Omega)}\\
				&\leq \|u_n-u_m\|^2_{W^{1,1}}\rightarrow 0, ~\text{as}~ n,m\rightarrow\infty.
			\end{split}
		\end{equation*}
		Therefore, there exists $v(x,\cdot)\in L^2(0,1)$ such that for any  $x\in (0,1)$,  we have
		\begin{equation}\label{proof3.1.6}
			\|u_n(x,\cdot)-v(x,\cdot)\|_{L^2(0,1)}\rightarrow 0, ~\text{as}~ n\rightarrow\infty.
		\end{equation}
		
		On the other hand, from \eqref{proof3.1.3} and  Fubini theorem, it follows that
		\begin{equation*}
			\begin{split}
				\int_0^1\left(\int_0^1|u_n(x,y)-u(x,y)|^2 dy\right)dx=\|u_n-u\|_{L^2(\Omega)}^2\rightarrow 0.
			\end{split}
		\end{equation*}
		This shows there exists a subsequence of  $\{u_n\}_{n\in\mathbb{N}}$, still denoted by itself, and  $E\subset (0,1)$ with $|E|=0$ such that for any  $x\in (0,1)-E$, 
		\begin{equation*}
			\|u_n(x,\cdot)-u(x,\cdot)\|_{L^2(0,1)}^2=\int_0^1|u_n(x,y)-u(x,y)|^2 dy\rightarrow 0  .
		\end{equation*}
		By \eqref{proof3.1.6}, for any $x\in (0,1)-E$, we have
		\begin{equation*}
			u(x,\cdot)=v(x,\cdot).
		\end{equation*}
		Since $|E\times[0,1]|=0$, so we can assume that $u=v$ in $\Omega$. Then we have
		\begin{equation*}
			u(0,\cdot)=\lim_{n\rightarrow\infty}u_n(0,\cdot)=0, 
		\end{equation*}
		\begin{equation*}
			u(1,\cdot)=\lim_{n\rightarrow\infty}u_n(1,\cdot)=0,
		\end{equation*}
		in $L^2(0,1)$.
		
		{\it Step 2.}
		From \eqref{proof3.1.5} and  Cauchy inequality, for any $y\in (0,1)$, we have
		\begin{equation*}
			\int_0^1|x^\frac{\alpha}{2}u_n(x,y)|^2 dx=\int_0^1\left|\int_0^y x^\frac{\alpha}{2}\partial_yu_n(x,s) ds\right|^2 dx\leq \iint_\Omega \left|x^\frac{\alpha}{2}\partial_yu_n(x,y)\right|^2 dydx,
		\end{equation*} 	
		and by \eqref{proof3.1.4} again, we hold
		\begin{equation*}
			\begin{split}
				\int_0^1\left|x^\frac{\alpha}{2}u_n(x,y)-x^\frac{\alpha}{2}u_m(x,y)\right|^2 dx
				&\leq \iint_\Omega \left|x^\frac{\alpha}{2}\partial_yu_n(x,y)-x^\frac{\alpha}{2}\partial_yu_m(x,y)\right|^2 dydx\\
				&=\|x^\frac{\alpha}{2}\partial_yu_n-x^\frac{\alpha}{2}\partial_yu_m\|_{L^2(\Omega)}^2\rightarrow 0,~\text{as}~n,m\rightarrow\infty.
			\end{split}
		\end{equation*}
		Thus, for any $y\in (0,1)$, there exists  $w(\cdot,y)\in L^2(0,1)$ such that
		\begin{equation}\label{proof3.1.7}
			\|x^\frac{\alpha}{2}u_n(\cdot,y)-w(\cdot, y)\|_{L^2(0,1)}^2=\int_0^1 |x^\frac{\alpha}{2}u_n(x,y)-w(x,y)|^2 dx\rightarrow 0, ~\text{as}~n\rightarrow\infty.
		\end{equation}	
		Moreover, from \eqref{proof3.1.3},
		\begin{equation*}
			\begin{split}
				\int_0^1\left(\int_0^1|x^\frac{\alpha}{2}u_n(x,y)-x^\frac{\alpha}{2}u(x, y)|^2 dx\right)dy
				&\leq \int_0^1\left(\int_0^1|u_n(x,y)-u(x, y)|^2 dx\right)dy\\
				&=\|u_n-u\|_{L^2(\Omega)}^2\rightarrow 0.
			\end{split}
		\end{equation*}
		Hence,  There exists a subsequence of $\{u_n\}_{n\in\mathbb{N}}$, still denoted by itself,  and $E\subset (0,1)$, with $ |E|=0$ such that for any $y\in (0,1)-E$, 
		\begin{equation*}
			\|x^\frac{\alpha}{2}u_n(\cdot,y)-x^\frac{\alpha}{2}u(\cdot,y)\|_{L^2(0,1)}^2=\int_0^1|x^\frac{\alpha}{2}u_n(x, y)-x^\frac{\alpha}{2}u(x,y)|^2 dx\rightarrow 0 .
		\end{equation*}
		Combining with  \eqref{proof3.1.7}, it follows that for any $y\in (0,1)-E$,
		\begin{equation*}
			x^\frac {\alpha}{2} u(\cdot,y)=w(\cdot,y).
		\end{equation*}
		
		Further, for any $\xi, \zeta\in (0,1)-E$, we have
		\begin{equation*}
			\begin{split}
				\|x^\frac{\alpha}{2}u(\cdot, \xi)-x^\frac{\alpha}{2}u(\cdot,\zeta)\|_{L^2(0,1)}
				\leq &\|x^\frac{\alpha}{2}u(\cdot, \xi)-x^\frac{\alpha}{2}u_n(\cdot, \xi)\|_{L^2(0,1)}\\
				&+\|x^\frac{\alpha}{2}u_n(\cdot, \xi)-x^\frac{\alpha}{2}u_n(\cdot, \zeta)\|_{L^2(0,1)}\\
				&+\|x^\frac{\alpha}{2}u_n(\cdot,\zeta)-x^\frac{\alpha}{2}u(\cdot, \zeta)\|_{L^2(0,1)}\\
				\leq &2\|u-u_n\|_{W^{1,1}}+\|x^\frac{\alpha}{2}u_n(\cdot, \xi)-x^\frac{\alpha}{2}u_n(\cdot, \zeta)\|_{L^2(0,1)},
			\end{split}
		\end{equation*}
		Notice that $|E\times[0,1]|=0$, so we can assume that $x^\frac{\alpha}{2}u=w$ in $\Omega$. Then we have
		\begin{equation*}
			x^\frac{\alpha}{2}u(\cdot,0)=\lim_{n\rightarrow\infty}x^\frac{\alpha}{2}u_n(\cdot,0)=0,
		\end{equation*}
		\begin{equation*}
			x^\frac{\alpha}{2}u(\cdot,1)=\lim_{n\rightarrow\infty}x^\frac{\alpha}{2}u_n(\cdot,1)=0.
		\end{equation*}
		in $L^2(0,1)$.
		Hence,
		\begin{equation*}
			u(\cdot,0)=0 ,  u(\cdot,1)=0 .
		\end{equation*}
		in $ L^2(0,1)$. The proof is completed.
	\end{proof}
	
	\begin{proposition}
		$(V(\Omega),(\cdot,\cdot)_V)$ is a  Hilbert space.	
	\end{proposition}
	\begin{proof}
		On one hand, it is easy to know that  $(V(\Omega),(\cdot,\cdot)_V)$ is a inner product space. On the other hand, if  $\{u_n\}_{n\in\mathbb{N}}\subset V(\Omega)$ is a Cauchy sequence, then
		\begin{equation*}
			\|u_n-u_m\|_{V}\rightarrow 0, ~\text{as}~n,m\rightarrow\infty.
		\end{equation*}
		Thus,
		\begin{equation*}
			\|u_n-u_m\|_{W^{1,1}}\rightarrow 0,\quad \|\partial_x\partial_y u_n-\partial_x\partial_yu_m\|_{L^2}\rightarrow 0, ~\text{as}~n,m\rightarrow\infty.
		\end{equation*}
		By Proposition \ref{3.2}, there exists $u\in W^{1,1}(\Omega;x^\alpha)$ and $w\in L^2(\Omega)$ such that
		\begin{equation*}
			\partial_x\partial_yu_n\rightarrow w~\text{strongly in}~  L^2(\Omega),
		\end{equation*}
		and 
		\begin{equation*}
			u_n\rightarrow u \text{ strongly in }  W^{1,1}(\Omega;x^\alpha).
		\end{equation*}
		Thus, for any $\varphi\in C_0^\infty(\Omega)$, we have
		\begin{equation*}
			\begin{split}
				\iint_\Omega w\varphi dxdy
				\leftarrow\iint_\Omega \left(\partial_x\partial_yu_n\right)\varphi dxdy
				&=\iint_\Omega u_n(\partial_x\partial_y\varphi)dxdy\\
				&\rightarrow\iint_\Omega u(\partial_x\partial_y\varphi)dxdy=\iint_\Omega (\partial_x\partial_yu)\varphi dxdy.
			\end{split}
		\end{equation*}
		Hence, $w=\partial_x\partial_yu$, and 
		\begin{equation*}
			\|u_n-u\|_{V}\rightarrow 0, ~\text{as}~  n\rightarrow\infty.
		\end{equation*}
		Therefore, $(V(\Omega), \|\cdot\|_{V(\Omega)})$ is a  Banach space. Further we can obtain that $(V(\Omega),(\cdot,\cdot)_V)$ is a  Hilbert space.
	\end{proof}
	
	\begin{remark}
		By the definition of $V_0(\Omega)$ and $W_0^{1,1}(\Omega;x^\alpha)$, we have $C_{0}^\infty(\Omega)\subset V_0(\Omega)$, and 
		$V_0(\Omega)$ is dense in $W_0^{1,1}(\Omega;x^\alpha)$.	
	\end{remark}
	
	\begin{lemma}\label{3.7}
		Let $u\in V_{0}(\Omega)$, then $(\partial_x u)(\cdot,0)=0$ in $L^2(0,1)$.
	\end{lemma}
	
	\begin{proof}
		Let  $u\in V_{0}(\Omega)$. Then there exists the sequence $\{u_n\}_{n\in\mathbb{N}}\subset C_{0}^\infty(\Omega)$ such that
		\begin{equation}\label{proof3.1.8}
			u_n\rightarrow u ,\quad \|u_n-u\|_{V}\rightarrow 0 \ \mbox{in} \  V_{0}(\Omega).
		\end{equation}
		For any  $n\in\mathbb{N}$, let  $w_n=\partial_xu_n$, then for any $y\in (0,1)$,
		\begin{equation*}
			w_n(x,y)=w_n(x,y)-w_n(x,0)=\int_0^y\partial_yw_n(x,s)ds.
		\end{equation*}
		By Cauchy inequality, we have
		\begin{equation*}
			\begin{split}
				\int_0^1|w_n(x,y)|^2 dx
				&=\int_0^1\left|\int_0^y\partial_yw_n(x,s)ds\right|^2 dx\\
				&\leq \iint_\Omega |\partial_yw_n(x,y)|^2 dxdy=\iint_\Omega |[\partial_y(\partial_xu_n)](x,y)|^2 dxdy.
			\end{split}
		\end{equation*}
		Hence, For any $y\in (0,1)$, when $n,m\rightarrow\infty$,
		\begin{equation*}
			\begin{split}
				\|w_n(\cdot,y)-w_m(\cdot,y)\|_{L^2(0,1)}
				&\leq \|[\partial_y(\partial_xu_n)](x,y)-[\partial_y(\partial_xu_m)](x,y)\|_{L^2(\Omega)}\\
				&\leq \|u_n-u_m\|_{V}\rightarrow 0.
			\end{split}
		\end{equation*}
		Thus, for any $y\in (0,1)$, there exists $w(\cdot, y)\in L^2(0,1)$ such that  
		\begin{equation*}
			w_n(\cdot, y)\rightarrow w(\cdot, y) \ \mbox{in} \ L^2(0,1) .
		\end{equation*}
		By \eqref{proof3.1.8}, we have
		\begin{equation*}
			\begin{split}
				\int_0^1\left(\int_0^1|w_n(x,y)-\partial_xu(x,y)|^2\ dx\right)dy
				&=\|\partial_xu_n-\partial_xu\|_{L^2(\Omega)}^2\rightarrow 0.
			\end{split}
		\end{equation*}
		So there exists a subsequence of $\{w_n\}_{n\in\mathbb{N}}$, still denoted by itself, and $E\subset (0,1)$ with $|E|=0$ such that for any  $y\in (0,1)-E$,
		\begin{equation*}
			w_n(\cdot,y)\rightarrow\partial_xu(\cdot,y)  \ \mbox{in} \ L^2(0,1).
		\end{equation*}
		This shows that for any $y\in (0,1)-E$,
		\begin{equation*}
			w(\cdot, y)=\partial_xu(\cdot, y) \ \mbox{in} \ L^2(0,1),
		\end{equation*}
		and $|E\times[0,1]|=0$, so for any   $(x,y)\in \Omega$ we can set $w=\partial_xu$, then
		\begin{equation*}
			\partial_xu(\cdot, 0)=\lim_{n\rightarrow\infty}\partial_xu_n(\cdot,0)=0 \ \mbox{in} \ L^2(0,1).
		\end{equation*}
		Thus we have completed the proof.
	\end{proof}
	
	\section{Well-posedness of the degenerate problem}
	In this section, we discuss the well-posedness of the degenerate problems \eqref{1.1}.
	To make the subsequent proofs easier, some equivalent forms of \eqref{3.2.9} are to be applied, and the following lemma will give equivalent definitions of several weak solutions.
	\begin{lemma}\label{3.9}
		Let $u\in W_{0}^{1,1}(\Omega;x^\alpha)$. For any $\varphi\in C_{0}^\infty(\Omega)$, we have
		\begin{equation}\label{3.2.10}
			\iint_\Omega \Big((x^\alpha\partial_yu)(\partial_y\varphi)+\frac{1}{2}(\partial_xu) \big[\partial_x(\partial_y\varphi)\big]\Big)dxdy=\iint_\Omega f(\partial_y\varphi)dxdy.
		\end{equation}	
	\end{lemma}
	\begin{proof}
		Let us prove this conclusion into two parts.
		
		{\it Step 1.}
		Let $u\in W_{0}^{1,1}(\Omega;x^\alpha)$ satisfying \eqref{3.2.9}. For any $\varphi\in C_{0}^\infty(\Omega)$, we have $\partial_y\varphi\in C_{0}^\infty(\Omega)$, Replacing \ $\varphi$ with $\partial_y\varphi$, it implies that \eqref{3.2.10} holds.
		
		{\it Step 2.}
		If $u\in W_{0}^{1,1}(\Omega;x^\alpha)$ satisfying \eqref{3.2.10},  for any  $\psi\in C_{0}^\infty(\Omega)$, let
		\begin{equation*}
			\varphi(x,y)=\int_0^y\psi(x,s)ds.
		\end{equation*}
		Then we have  $\varphi\in C_{0}^\infty(\Omega)$, which can be taken as the test function in \eqref{3.2.10}. Thus,
		\begin{equation*}
			\begin{split}
				\iint_\Omega f\psi dxdy
				&=
				\iint_\Omega f(\partial_y\varphi)dxdy\\
				&=\iint_\Omega \left((x^\alpha\partial_yu)(\partial_y\varphi)+\frac{1}{2}(\partial_xu) \big[\partial_x(\partial_y\varphi)\big]\right)dxdy \\
				&=\iint_\Omega \left((x^\alpha\partial_yu)\psi+\frac{1}{2}(\partial_xu)(\partial_x\psi)\right)dxdy.
			\end{split}
		\end{equation*}
		By the arbitrariness of $\psi$, we can deduce \eqref{3.2.9}. Thus, $u$ is a weak solution of equation \eqref{1.1}. The proof is completed.
	\end{proof}
	
	\begin{lemma}\label{3.10}
		Let $u\in W_{0}^{1,1}(\Omega;x^\alpha)$. For any $\varphi\in C_{0}^\infty(\Omega)$, we have
		\begin{equation}\label{3.2.11}
			\begin{split}
				\iint_\Omega \left((x^\alpha\partial_yu)(\partial_y\varphi)+\frac{1}{2}(\partial_xu)\big[ \partial_x(\partial_y\varphi)\big]\right)e^{-\theta y} dxdy
				=\iint_\Omega f(\partial_y\varphi)e^{-\theta y}dxdy.
			\end{split}
		\end{equation}
	\end{lemma}
	\begin{proof}
		On one hand, let $u\in W_{0}^{1,1}(\Omega;x^\alpha)$ satisfying \eqref{3.2.9}. For any $\varphi\in C_{0}^\infty(\Omega)$, we have $(\partial_y\varphi)e^{-\theta y}\in C_{0}^\infty(\Omega)$, Replacing \ $\varphi$ with $(\partial_y\varphi)e^{-\theta y}$, it shows that \eqref{3.2.11} holds.
		
		On the other hand, if for any $\varphi\in C_{0}^\infty(\Omega)$, $u\in W_{0}^{1,1}(\Omega;x^\alpha)$ satisfying \eqref{3.2.10}. For any  $\psi\in C_{0}^\infty(\Omega)$, let
		\begin{equation*}
			\varphi(x,y)=\psi(x,y)e^{\theta y}-\theta \int_0^y\psi(x,s)e^{\theta s} ds.
		\end{equation*}
		Then we have  $\varphi\in C_{0}^\infty(\Omega)$, which can be taken as the test function in \eqref{3.2.11}. Therefore,
		\begin{equation*}
			\iint_\Omega \Big((x^\alpha\partial_yu)(\partial_y\varphi)+\frac{1}{2}(\partial_xu) \big[\partial_x(\partial_y\varphi)\big]\Big)dxdy=\iint_\Omega f(\partial_y\varphi)dxdy.
		\end{equation*}
		by the arbitrariness of $\psi$, and Lemma \ref{3.9}, we can deduce \eqref{3.2.9}. Thus, $u\in W_{0}^{1,1}(\Omega;x^\alpha)$ is a weak solution of equation \eqref{1.1}. The proof is completed.
	\end{proof}
	To prove the existence of weak solutions, we first show the following lemmas.
	\begin{lemma}\label{3.11}
		Let $(H, (\cdot,\cdot)_H)$ be a Hilbert space, and $(V, |\cdot|_V)$ be a dense subspace of $H$. Let $T: V\rightarrow H$ be a bounded linear operator, and let $T^{-1}: R(T)\rightarrow V$ exist and be bounded with respect to the norm $|\cdot|_A$. Then, the range of the adjoint operator $T^*$ is the entire space $H$, i.e., $R(T^*)=H$.
		
	\end{lemma}
	
	\begin{proof}
		It suffices to prove that for any given $h\in H$, there exists $u\in H$ such that $T^*u=h$. To this end, we define a linear functional
		\begin{equation*}
			F: R(T)=D(T^{-1})\rightarrow \mathbb{R}  , \ z\mapsto (h,T^{-1}z)_H,
		\end{equation*}
		then
		\begin{equation*}
			\|F\|=\sup_{\|z\|=1}|F(z)|\leq \|h\|\cdot \|T^{-1}\|.
		\end{equation*}
		Thus, $F$ is a bounded linear functional on $R(T)$. We extend $F$ to a bounded linear functional on $\overline{R(T)}$. Moreover, we know that $\overline{R(T)}\subset H$ is a Hilbert space. Therefore, by the Riesz representation theorem, there exists $u\in \overline{R(T)}$ such that
		\begin{equation*}
			(u,z)_H=F(z)=(h,T^{-1}z)_H, \ z\in R(T).
		\end{equation*}
		Thus
		\begin{equation*}
			(u,Ty)_H=(h,y)_H, \ y\in V,
		\end{equation*}
		i.e.
		\begin{equation*}
			\langle T^*u,y\rangle_{V^*,V}=\langle u,Ty\rangle_{H^*,H}=(u,Ty)_H=(h, y)_H, \ y\in V.
		\end{equation*}
		Hence, due to the density of $V$ in $H$, it follows that $T^*u$ is also a bounded linear operator on $H$. Moreover,
		\begin{equation*}
			(T^*u, y)_H=(h,y)_H, \ y\in H.
		\end{equation*}
		Therefore, $T^*u=h$. Hence, we have $R(T^*)=H$. The proof is completed.
	\end{proof}
	\begin{lemma}\label{3.12L}
		Let $(H,(\cdot,\cdot)_H)$ be a Hilbert space, and $(V,|\cdot|_V)$ be a dense subspace of $H$. Let $a: H\times V\rightarrow \mathbb{R}$ be a bilinear functional satisfying the following conditions:
		\begin{enumerate}
			\item [\rm (1)]
			There exists a constant $M>0$ such that for any $u\in H$ and $v\in V$, we have
			\begin{equation*}
				|a(u,v)|\leq M \|u\|_H \cdot \|v\|_V;
			\end{equation*}	
			
			\item [\rm (2)]
			There exists a constant $\delta >0$ such that for any $v\in V$, we have
			\begin{equation*}
				a(v,v)\geq \delta \|v\|_H^2.
			\end{equation*}
		\end{enumerate}	
		For any bounded linear functional $F: H\rightarrow \mathbb{R}$, there exists a unique $u\in H$ such that for any $v\in V$, we have
		\begin{equation*}
			F(v)=a(u,v).
		\end{equation*}
	\end{lemma}	
	
	\begin{proof}
		From condition (1),  it follows that for any fixed $v\in V$, $a(\cdot,v): H\rightarrow \mathbb{R}$ is a bounded linear functional. Therefore, by the Riesz representation theorem, there exists a unique $Av\in H$ such that
		\begin{equation}\label{3.2.12}
			a(u,v)=(u,Av), \ u\in H.
		\end{equation}
		
		On one hand, due to the bilinearity of $a(u,v)$ and condition (1), it is evident that the defined $A: V\rightarrow H$ is a bounded linear operator. On the other hand, from condition (2) and equation (3.2.12), it can be deduced that for any $v\in V$, we have
		\begin{equation*}	
			(v,Av)_H\geq \delta\|v\|_H^2.		
		\end{equation*}
		This implies that for any $v\in V$, we have
		\begin{equation*}
			\|Av\|_H\geq \delta \|v\|_H,	
		\end{equation*}
		Hence, the inverse operator
		\begin{equation*}
			A^{-1}: R(A)\rightarrow V
		\end{equation*}
		exists, and it is a bounded linear operator in the norm $|\cdot|_H$ sense. From Lemma \ref{3.11}, we know that $R(A^*)=H$. Since $F$ is a bounded linear functional on $H$, by the Riesz representation theorem, for any $v\in H$, there exists a unique $h\in H$ such that
		\begin{equation*}
			F(v)=(h,v)_H.
		\end{equation*}
		Moreover, since $R(A^*)=H$, for any $h\in H$, there exists $u\in H$ such that $A^*u=h$. Therefore, for any $v\in V$, we have
		\begin{equation*}
			(u, Av)_H=(A^*,v)_H=(h,v)_H.
		\end{equation*}
		This, together with \eqref{3.2.12}, we obtain
		\begin{equation*}
			F(v)=a(u,v).
		\end{equation*}
		Thus, the above lemma is proven.
	\end{proof}
	
	Now we are ready to prove Theorem~\ref{3.12}.
	
	\begin{proof}[\textbf{Proof of Theorem~\ref{3.12}}]
		At first, let
		\begin{equation*}
			a(u,v)=\iint_\Omega\left((x^\alpha\partial_yu)(\partial_yv)+\frac{1}{2}(\partial_xu)(\partial_x\partial_yv)\right)e^{-\theta y}dxdy,\ u\in W_{0}^{1,1}(\Omega;x^\alpha), v\in V_0(\Omega),
		\end{equation*}
		where $\theta>0$ is a given constant. Then we have
		\begin{equation}\label{3.2.13}
			|a(u,v)|\leq \|u\|_{W^{1,1}}\cdot \|v\|_{V}
		\end{equation}
		for any $u\in W_{0}^{1,1}(\Omega;x^\alpha)$ and $v\in V(\Omega)$.
		For any $v\in V(\Omega)$, there exists
		\begin{equation*}
			\begin{split}
				\iint_\Omega (\partial_x v)(\partial_x\partial_yv)e^{-\theta y}dxdy
				&=\frac{1}{2}\iint_\Omega e^{-\theta y}\partial_y|\partial_xv|^2dxdy\\
				&=\frac{1}{2}\iint_\Omega \partial_y\left(|\partial_xv|^2e^{-\theta y}\right)dxdy+\frac{\theta}{2}\iint_\Omega |\partial_xv|^2e^{-\theta y}dxdy\\
				&=\frac{e^{-\theta }}{2}\int_0^1|\partial_xv|^2\Big|_{y=1}dx-\frac{1}{2}\int_0^1|\partial_xv|^2\Big|_{y=0}dx \\
				&+\frac{\theta}{2}\iint_\Omega|\partial_xv|^2 e^{-\theta y}dxdy. 
			\end{split}
		\end{equation*}
		By Lemma \ref{3.7}, it is known that
		\begin{equation*}
			\int_0^1|\partial_xv|^2\Big|{y=0}dx=0.
		\end{equation*}
		Hence,
		\begin{equation*}
			\iint_\Omega (\partial_x v)(\partial_x\partial_yv)e^{-\theta y}dxdy\geq \frac{\theta e^{-\theta}}{2}\iint_\Omega |\partial_x v|^2 dxdy.
		\end{equation*}
		
		On the other hand, since $C_{0}^\infty(\Omega)$ is dense in $V_{0}(\Omega)$, the following Poincaré inequality holds:
		\begin{equation*}
			\iint_\Omega v^2 dxdy\leq \mu \iint_\Omega |\partial_x v|^2dxdy.
		\end{equation*}
		Therefore,
		\begin{equation*}
			\iint_\Omega (\partial_x v)(\partial_x\partial_yv)e^{-\theta y}dxdy\geq \frac{\theta e^{-\theta }}{4}\iint_\Omega |\partial_x v|^2dxdy+\frac{\theta e^{-\theta }}{4\mu}\iint_\Omega v^2dxdy.
		\end{equation*}
		From the definition of $a(u,v)$, it follows that for any $v\in V_{0}(\Omega)$,
		\begin{equation}\label{3.2.14}
			|a(v,v)|\geq \delta \|v\|_{W^{1,1}(\Omega;x^\alpha)}^2,
		\end{equation}
		where
		\begin{equation*}
			\delta=\min\left\{e^{-\theta }, \frac{\theta e^{-\theta }}{8}, \frac{\theta e^{-\theta }}{8\mu}\right\}.
		\end{equation*}
		Let
		\begin{equation*}
			H=W_{0}^{1,1}(\Omega;x^\alpha),\quad V=V_{0}(\Omega).
		\end{equation*}
		Then $V$ is dense in $H$, and it can be seen from equations \eqref{3.2.13} and \eqref{3.2.14} that the conditions (1) and (2) in Lemma \ref{3.12L} are satisfied. It is evident that
		\begin{equation*}
			\begin{split}
				\iint_\Omega f\partial_yv e^{-\theta y}dxdy
				&\leq\left(\iint_\Omega\left|f\right|^2 x^{-\alpha}
				dxdy \right)^{\frac{1}{2}}\left(\iint_\Omega \left|\partial_yv \right|^2 x^{\alpha} dxdy\right)^{\frac{1}{2}}\\
				&=\left||f|\right|_{L^2(\Omega,x^{-\alpha})}\left||v|\right|_{W^{1,1}(\Omega,x^{\alpha})}.
			\end{split}
		\end{equation*}	
		That is, $\iint_\Omega f\partial_yv e^{-\theta y}dxdy$ is a bounded linear functional on $H$ with respect to $v$. Therefore, there exists a unique $u\in H$ such that
		\begin{equation}\label{36}
			a(u,v)=\iint_\Omega f\partial_yv e^{-\theta y}dxdy, \ \forall v\in V_{0}(\Omega).
		\end{equation}
		In other words, for any $v\in V_{0}(\Omega)$, we have
		\begin{equation}\label{proof3.7}
			\iint_\Omega\left((x^\alpha\partial_yu)(\partial_yv)+\frac{1}{2}(\partial_xu)(\partial_x\partial_yv)\right)e^{-\theta y}dxdy=\iint_\Omega f\partial_yv e^{-\theta y}dxdy.
		\end{equation}
		This along with Lemma \ref{3.10}, shows that $u$ is a weak solution of the equation \eqref{1.1}.
		
		Next, by \eqref{36}, we have
		\begin{align*}
			a(u,u)&=\iint_\Omega f\partial_yu e^{-\theta y}dxdy=\iint_\Omega x^{-\frac{\alpha}{2}}f x^{\frac{\alpha}{2}}\partial_yu e^{-\theta y}dxdy\\
			&\leq \left( \iint_\Omega x^{-\alpha}f^2 dxdy\right)^\frac{1}{2}\left( \iint_\Omega x^{\alpha}(\partial_yu)^2 dxdy\right)^\frac{1}{2}\\
			&=\left||f|\right|_{L^2(\Omega,x^{-\alpha})}\left||u|\right|_{W^{1,1}(\Omega,x^{\alpha})}.
		\end{align*}
		This, along with \eqref{3.2.14} replacing $v$ by $u$, it implies that the desired inequality \eqref{ine44} holds. The proof is completed.
	\end{proof}

	\section{The existence of a Stackelberg-Nash equilibrium}
	In this section, we study the Stackelberg-Nash game problem {\bf (P)}. At first, we need some auxiliary results.
	
	To facilitate our discussion, we will employ fundamental definitions and concepts. For a comprehensive understanding of these concepts, we refer the
	readers to references such as section 3 in \cite{sun2022fundamental} or \cite{heinonen2018nonlinear,muckenhoupt1972weighted,mamedov2021poincare}.
	
	Let $v(x)$ and $w(x)$ be non-negative and local integrable with respect to Lebesgue measure and a measurable set $E\subset\R^2$. 
	
	Denote
	$$
	B:=B(x,t):=\{y\in\R^2:|y-x|<t\},
	$$
	with $x\in\Omega$, $0<t\leq5~\text{diam}(\Omega)$.  
	\begin{definition}\label{A}
		We say that
		the function $v(x)$ belongs to the Muckenhoupt class $A_\infty$  if there exist constants $C$
		and $\delta> 0$ such that for any measurable set $E \subset B$,
		\begin{equation}\label{22}
			\frac{v(E)}{v(B)}\leq C\left( \frac{|E|}{|B|} \right)^\delta,    
		\end{equation}
		where $v(\cdot)=\int_\cdot v(x)dx$.
	\end{definition}
	\begin{definition}
		We say that
		the function $w(x)$ belongs to the Muckenhoupt class $A_p,1\leq p<\infty$  if there exist constants $C_p$ such that \begin{equation*}
			\begin{split}
				\begin{cases}
					\left(|B\cap\Omega|^{-1}\int_{B\cap\Omega} w(x)dx\right) \left(|B\cap\Omega|^{-1}\int_{B\cap\Omega} w^{-\frac{1}{p-1}}(x)dx\right)^{p-1}\leq C_p<\infty, &~\text{if}~1<p<\infty,\\
					\left(|B\cap\Omega|^{-1}\int_{B\cap\Omega} w(x)dx\right) \leq C_1\essinf_{x\in B\cap\Omega}w(x), &~\text{if}~p=1.
				\end{cases}
			\end{split}	
		\end{equation*}
	\end{definition}
	
	The following result, which can be found in section 3 in \cite{sun2022fundamental}
	\begin{lemma}\label{cite}
		Let $1\leq p\leq q<\infty$. Suppose that $v\in A_\infty$, $w_j(x)\in A_p, j=1,2$.  If the condition    \begin{equation}
			\label{222}
			|B|^{-1} {\rm diam} (B) v(B\cap\Omega)^{\frac{1}{q}} \left[ w_j^{-\frac{1}{p-1}}(B\cap\Omega)\right]^{\frac{p-1}{p}} \leq A_{pq}<\infty,
		\end{equation}
		is fulfilled for any Euclidean ball $B(x, t)$ having a center $x$, $0<t\leq\text{diam}(\Omega)$. Then there exists
		a positive number $C_0$ depending on $q$ and the constants $C$ and $\delta$ occurring in \eqref{22},  such that for
		any $u\in {\rm Lip}_0(\Omega)$, we have
		\begin{equation}\label{2222}
			\left(\int_{\Omega} |u|^qvdx\right)^{\frac{1}{q}}\leq C_0A_{pq}\sum_{j=1}^2\left(\int_{\Omega}|\partial_{x_j}  u|^pw_jdx\right)^{\frac{1}{p}},
		\end{equation}
		where ${\rm Lip}_0(\Omega)$ is the class of Lipschitz continuous functions which have compact support in $\Omega$.
	\end{lemma}

	\begin{lemma}\label{neicha4.2}
		Let  $u\in W_{0}^{1,1}(\Omega;x^\alpha)$ , and  $2\le q \le 4$, one has	
		\begin{equation}\label{4.2.24}
			\|u\|_{L^q(\Omega)}\le C\|u\|_{W_{0}^{1,1}(\Omega;x^\alpha)}.	
		\end{equation}	
	\end{lemma}	
	\begin{proof}
		We divide the proof into two steps.
		
		{\it Step 1.}
		By the definition of $W_{0}^{1,1}(\Omega;x^\alpha)$, we have $$\overline{C_0^\infty(\Omega)}^{W^{1,1}(\Omega;x^\alpha)} \subset \overline{{\rm Lip}_0(\Omega)}^{W^{1,1}(\Omega;x^\alpha)}. $$ 
		We just need to prove 
		$$
		\overline{{\rm Lip}_0(\Omega)}^{W^{1,1}(\Omega;x^\alpha)} \subset \overline{C_0^\infty(\Omega)}^{W^{1,1}(\Omega;x^\alpha)}.
		$$
		
		To this end, for all $u\in\overline{{\rm Lip}_0(\Omega)}^{W^{1,1}(\Omega;x^\alpha)}$, we can choose the sequence $\{u_n\}_{n\in \mathbb{N}} \in {\rm Lip}_0(\Omega)$ such that $$\|u_n-u\|_{W^{1,1}(\Omega;x^\alpha)}\rightarrow 0 ~\text{as}~n \rightarrow \infty, ~\text{and ~} |Du_n|\leq C  ~\text{a.e. in}~ \Omega.$$ Thus $u_n\in H_0^1(\Omega)$, and there exists $\{v_n\}_{n\in\mathbb{N}}\subset C_0^\infty(\Omega)$ so that
		\begin{equation*}
			\|v_n-u_n\|_{H_0^1(\Omega)}\rightarrow 0, ~\text{as}~n\rightarrow \infty,
		\end{equation*}
		and
		\begin{equation*}
			\|v_n-u_n\|_{W^{1,1}(\Omega;x^\alpha)}\leq C\|v_n-u_n\|_{H^1(\Omega)}\rightarrow 0, ~\text{as}~ n\rightarrow \infty.
		\end{equation*}
		From these, we have
		\begin{equation*}
			\|v_n-u\|_{W^{1,1}(\Omega;x^\alpha)}\leq \|v_n-u_n\|_{W^{1,1}(\Omega;x^\alpha)}+\|u_n-u\|_{W^{1,1}(\Omega;x^\alpha)}
			\rightarrow 0 ~\text{as}~ n\rightarrow \infty.
		\end{equation*}
		
		It is easy to deduce that $W_{0}^{1,1}(\Omega;x^\alpha)=\overline{{\rm Lip}_0(\Omega)}^{W^{1,1}(\Omega;x^\alpha)}$, and ${\rm Lip}_0(\Omega)$ is dense in  $W_{0}^{1,1}(\Omega;x^\alpha)$.
		
		{\it Step 2.}
		One can verify that $v(x,y)=1$ satisfying the definition of $ A_\infty(\Omega)$ (see Definition \ref{A}). Then we select $w_1=1$, $w_2=x^{\alpha}$, and $p=2$ in \eqref{222} of Lemma \ref{cite}. If $p=2$, $2\le q \le 4$, one can obtain
		\begin{equation*}
			\begin{split}
				|B|^{-1} {\rm diam} (B)  v(B\cap\Omega)^{\frac{1}{q}} &\left({\iint_{B\cap \Omega} w_1(x,y)^{-1}dxdy}\right)^{\alpha} \leq Cr^\frac{2}{q}	\le C,\\
				|B|^{-1} {\rm diam} (B)  v(B\cap\Omega)^{\frac{1}{q}} &\left({\iint_{B\cap \Omega} w_2(x,y)^{-1}dxdy}\right)^{\alpha}\le C.
			\end{split}	
		\end{equation*}
		Further, from Lemma \ref{cite}, we know that for arbitrary $u \in W_{0}^{1,1}(\Omega;x^\alpha)$,
		\begin{equation*}
			\left(\iint_{\Omega} |u|^qdxdy\right)^{\frac{1}{q}}\le C \left[\left(\iint_{\Omega}|\partial_x  u|^2dxdy\right)^{\alpha}+\left(\iint_{\Omega}|\partial_{y}u|^2 x^{\alpha}dxdy\right)^{\alpha}\right]. 	
		\end{equation*}
		By \ {\rm Jensen} inequality $\frac{x^t+y^t}{2}\le  \left(\frac{x+y}{2} \right)^t$,  $t\in (0,1)$, $x,y \in \mathbb{R^+}$, we have
		\begin{equation*}
			\begin{aligned}
				&\left(\iint_{\Omega}|\partial_x u|^2dxdy\right)^{\alpha}+\left(\iint_{\Omega}  |\partial_y u|^2 x^{\alpha} dxdy\right)^{\alpha}\\
				\le &C\left(\iint_{\Omega}\left(|\partial_x u|^2+|\partial_y u|^2 x^{\alpha}\right)dxdy\right)^{\alpha}.
			\end{aligned}
		\end{equation*}
		The proof is completed.
	\end{proof}

	\begin{lemma}\label{R4.3}
		For all  $u\in W_{0}^{1,1}(\Omega;x^\alpha)$, let
		\begin{equation*}
			\tilde u(x,y)=
			\begin{cases}
				u(z), & z\in \Omega,\\
				0, & z\in \Omega^c.
			\end{cases}
		\end{equation*}
		Then $\tilde{u}\in W_{0}^{1,1}(\mathbb{R}^2;x^\alpha)$.	
		
	\end{lemma}
	\begin{proof}
		For all  $u\in W_{0}^{1,1}(\Omega;x^\alpha)$, there exists $\{u_n\}_{n\in\mathbb{N}}\subset C_0^\infty(\Omega)$, such that 
		\begin{equation*}
			u_n\rightarrow u \text{ strongly in }  W^{1,1}(\Omega;x^\alpha).
		\end{equation*} 	
		Then for all $\varphi\in C_0^\infty(\Omega)$, we have 
		\begin{equation*}
			\begin{split}
				\iint_{\mathbb{R}^2} (x^{\frac{\alpha}{2}}\partial_y\tilde u)\varphi dxdy
				&=-\iint_{\mathbb{R}^2}\tilde u(x^{\frac{\alpha}{2}}\partial_y\varphi)dxdy=-\iint_\Omega u(x^{\frac{\alpha}{2}}\partial_y\varphi)dxdy\\
				&=-\lim_{n\rightarrow\infty}\iint_\Omega u_n(x^{\frac{\alpha}{2}}\partial_y\varphi)dxdy=\lim_{n\rightarrow\infty}\iint_\Omega (x^{\frac{\alpha}{2}}\partial_y u_n)\varphi dxdy\\
				&=\iint_\Omega (x^{\frac{\alpha}{2}}\partial_y u)\varphi dxdy=\iint_{\mathbb{R}^2}{\left[x^{\frac{\alpha}{2}}\partial_y u\right]}^\sim\varphi dxdy
			\end{split}
		\end{equation*}
		in the sense of distribution. Thus we have
		\begin{equation*}
			x^{\frac{\alpha}{2}}\partial_y\tilde u={\left[x^{\frac{\alpha}{2}}\partial_y u\right]}^\sim ,
		\end{equation*}
		where
		\begin{equation*}
			{\left[x^{\frac{\alpha}{2}}\partial_y u\right]}^\sim(z)=
			\begin{cases}
				x^{\frac{\alpha}{2}}\partial_y u(z), & z\in \Omega,\\
				0, & z\in \Omega^c.
			\end{cases}
		\end{equation*}
		In a similar way,
		\begin{equation*}
			\partial_{x}\tilde u=(\partial_{x}u)^\sim.
		\end{equation*}	
		Hence $\tilde u\in W^{1,1}_0(\mathbb{R}^2,x^{\alpha})$ and
		\begin{equation*}
			\|\tilde u\|_{W^{1,1}_0(\mathbb{R}^2,x^{\alpha})}=\|u\|_{W^{1,1}_0(\Omega,x^{\alpha})}.
		\end{equation*}	
		
	\end{proof}	
	
	\begin{proposition}
		\label{compactembed}
		The embedding $W_{0}^{1,1}(\Omega;x^\alpha) \hookrightarrow L^2(\Omega)$ is compact.
	\end{proposition}
	
	\begin{proof}
		We divide the proof into following steps.
		
		{\it Step 1.} Let $\{u_n\}_{n\in\mathbb{N}}\subset W_{0}^{1,1}(\Omega;x^\alpha)$ is a bounded sequence. We shall show that the sequence $\{\tilde u_n\}_{n\in\mathbb{N}}$ is precompact in  $L^1(\Omega)$.
		
		Let
		\begin{equation*}
			\tilde u_n(x,{y})=
			\begin{cases}
				u_n(z), & z\in \Omega,\\
				0, & z\in \Omega^c.
			\end{cases}
		\end{equation*}
		By Lemma {\rm\ref{R4.3}}, we have $\tilde{u}\in W_{0}^{1,1}(\mathbb{R}^2)$. Next   According to Theorem {\rm 2.32} in
		{\rm \cite{adams2003sobolev}}, it suffices to prove that for all $\epsilon>0$, there exists $\delta>0$ and  $G\subset\subset \Omega$, such that
		\begin{equation}\label{4.2.25}
			\int_\Omega|\tilde u_n(z+\xi)-\tilde u_n(z)|dz<\epsilon,\quad \int_{\Omega-\bar G}|\tilde u_n(z)|dz<\epsilon,
		\end{equation}
		where $\xi=(\xi_1,\xi_2)\in\mathbb{R}^2$ with $|\xi|<\delta$. Let $\epsilon\in (0,1)$. by the definition of $W_{0}^{1,1}(\Omega;x^\alpha)$,  for any $n\in\mathbb{N}$, there exists $\hat u_n\in C_0^\infty(\Omega)$ such that
		\begin{equation}\label{4.2.26}
			\|u_n-\hat u_n\|_{W^{1,1}}<\epsilon.
		\end{equation}
		Still denoted by $\widetilde{(\hat u_n)}=\hat u_n$ for each $\hat u_n\in C_0^\infty(\Omega)$, Obviously, $\{\hat u_n\}_{n\in\mathbb{N}}\subset C_0^\infty(\Omega)$  is bounded in $W_0^{1,1}(\mathbb{R}^2;x^\alpha)$, and
		\begin{equation*}
			\begin{split}
				\hat u_n(x,y)
				&=\int_0^{y}(\partial_{y}\hat u_n)(x,t)dt=\int_0^{y} x^{-\frac{\alpha}{2}}\left(x^{\frac{\alpha}{2}}\partial_y\hat u_n\right)(x,t)dt.
			\end{split}
		\end{equation*}
		Hence
		\begin{equation}\label{4.2.27}
			\begin{split}
				\int_0^1|\hat u_n(s,y)| ds
				&\leq\int_0^1 \left|\int_0^y s^{-\frac{\alpha}{2}} s^{\frac{\alpha}{2}}\partial_y\hat u_n(s,t)dt\right|ds\\
				&\leq \|x^{-\alpha}\|_{L^1(0,1)}^{\frac{1}{2}} \|x^{\frac{\alpha}{2}}\partial_y\hat u_n\|_{L^2(\Omega)}y^{\frac{1}{2}}.
			\end{split}
		\end{equation}
		It further follows that for any $j\in\mathbb{N}$,
		\begin{equation*}
			\begin{split}
				\iint_{(0,1)\times (0,j^{-1})}|\hat u_n(x,y)|dxdy
				&\leq \|x^{-\alpha}\|_{L^1(0,1)}^{\frac{1}{2}}\|x^{\frac{\alpha}{2}}\partial_y\hat u_n\|_{L^2(\Omega)}\int_0^{j^{-1}} y^{\frac{1}{2}} dy\\
				&=\frac{1}{1+\frac{1}{2}}\|x^{-\alpha}\|_{L^1(0,1)}^{\frac{1}{2}}\|x^{\frac{\alpha}{2}}\partial_{y}\hat u_n\|_{L^2(\Omega)}j^{-(\frac{1}{2}+1)}.
			\end{split}
		\end{equation*}
		Similar to the above argument, we can obtain
		\begin{equation*}
			\begin{split}
				\iint_{(0,1)\times(1-j^{-1}, 1)}|\hat u_n(x,y)|dxdy
				&\leq \frac{1}{1+\frac{1}{2}}\|x^{-\alpha}\|_{L^1(0,1)}^{\frac{1}{2}}\|x^{\frac{\alpha}{2}}\partial_{y}\hat u_n\|_{L^2(\Omega)} \left(1-(1-\frac{1}{j})^{(\frac{1}{2}+1)}\right),\\
				\iint_{(0,j^{-1})\times (0,1)}|\hat u_n(x,y)|dxdy
				&\leq\frac{1}{1+\frac{1}{2}}\|\partial_{x}\hat u_n\|_{L^2(\Omega)}j^{-(\frac{1}{2}+1)},\\
				\iint_{(1-j^{-1},1)\times (0,1)}|\hat u_n(x,y)|dxdy
				&\leq\frac{1}{1+\frac{1}{2}}\|\partial_{x}\hat u_n\|_{L^2(\Omega)}\left(1-(1-\frac{1}{j})^{(\frac{1}{2}+1)}\right).
			\end{split}
		\end{equation*}
		Denote
		\begin{equation*}
			\Omega_j=\{(x,y)\in \Omega\mid x,y\in (0,j^{-1})\cup(1-j^{-1}, 1)\}, \ j\in\mathbb{N}.
		\end{equation*}
		Then
		\begin{equation*}
			\int_{\Omega_j}|\hat u_n(z)|dz\leq C\|\hat u_n\|_{W_0^{1,1}(\Omega,x^{\alpha})}\max\left\{j^{-(\frac{1}{2}+1)},\left(1-(1-\frac{1}{j})^{(\frac{1}{2}+1)}\right)\right\},
		\end{equation*}
		where $C>0$ is a constant depending only on $x^{\alpha}$. This implies that there exists  $j_0\in\mathbb{N}$ such that when  $j\geq j_0$, it holds
		\begin{equation*}
			\int_{\Omega_j}|\hat u_n(z)|dz<\epsilon.
		\end{equation*}
		Now let $G=\Omega-\bar \Omega_{j_0}$, we have proved that
		\begin{equation}\label{4.2.28}
			\int_{\Omega-\bar G}|\hat u_n(z)|dz<\epsilon.	
		\end{equation}
		
		Let $\xi=(\xi_1,\xi_2)\in\mathbb{R}^2$, without loss of generality, we assume $\xi_1>0, \xi_2>0$. Noting that
		\begin{equation*}
			\begin{split}
				|\hat u_n(x+\xi_1,y+\xi_2)-\hat u_n(x,y+\xi_2)|
				&=\left|\int_{x}^{x+\xi_1}\partial_{x}\hat u_n(s,{y}+\xi_2)ds\right|\\
				&\leq \xi_1^{\frac{1}{2}}\left(\int_{x}^{x+\xi_1}\left|\partial_{x}\hat u_n(s,{y}+\xi_2)\right|^2 ds\right)^\frac{1}{2}\\
				&\leq \xi_1^{\frac{1}{2}}\left(\int_{\mathbb{R}}\left|\partial_{x}\hat u_n(s,{y}+\xi_2)\right|^2 ds\right)^\frac{1}{2},
			\end{split}
		\end{equation*}
		we have
		\begin{equation*}
			\int_{\mathbb{R}}|\hat u_n(x+\xi_1,{y}+\xi_2)-\hat u_n(x,{y}+\xi_2)|dy\leq \xi_1^{\frac{1}{2}}\|\partial_{x}\hat u_n\|_{L^2(\Omega)}.
		\end{equation*}
		Hence
		\begin{equation}\label{4.2.buchong}
			\iint_{\mathbb{R}^2}|\hat u_n(x+\xi_1,{y}+\xi_2)-\hat u_n(x,{y}+\xi_2)|dxdy\leq \xi_1^{\frac{1}{2}}\|\partial_{x}\hat u_n\|_{L^2(\Omega)}.
		\end{equation}
		Similarly, we obtain that
		\begin{equation*}
			\int_{\mathbb{R}}|\hat u_n(x,{y}+\xi_2)-\hat u_n(x,{y})|dx\leq \xi_2^\frac{1}{2}\|x^{-\alpha}\|_{L^1(0,1)}^\frac{1}{2}\|x^{\frac{\alpha}{2}}\partial_{y}\hat u_n\|_{L^2(\Omega)}.
		\end{equation*}
		Thus
		\begin{equation}\label{4.2.29}
			\iint_{\mathbb{R}^2}|\hat u_n(x,{y}+\xi_2)-\hat u_n(x,{y})|dxdy\leq \xi_2^\frac{1}{2}\|x^{-\alpha}\|_{L^1(0,1)}^\frac{1}{2}\|x^{\frac{\alpha}{2}}\partial_{y}\hat u_n\|_{L^2(\Omega)}.
		\end{equation}
		From
		\begin{equation*}
			\begin{split}
				|\hat u_n(z+\xi)-\hat u_n(z)|
				&\leq |\hat u_n(x+\xi_1,{y}+\xi_2)-\hat u_n(x, {y}+\xi_2)|+|\hat u_n(x,{y}+\xi_2)-\hat u_n(x,{y})|
			\end{split}
		\end{equation*}
		and \eqref{4.2.buchong}, \eqref{4.2.29}, we have
		\begin{equation*}
			\int_{\mathbb{R}^2}|\hat u_n(z+\xi)-\hat u_n(z)|dz\leq \xi_1^{\frac{1}{2}}\|\partial_{x}\hat u_n\|_{L^2(\Omega)}+\xi_2^\frac{1}{2}\|x^{-\alpha}\|_{L^1(0,1)}^\frac{1}{2}\|x^{\frac{\alpha}{2}}\partial_{y}\hat u_n\|_{L^2(\Omega)}.
		\end{equation*}
		Thus, when $|\xi|=\sqrt{\xi_1^2+\xi_2^2}$ is sufficiently small, we have
		\begin{equation}\label{4.2.30}
			\int_{\mathbb{R}^2}|\hat u_n(z+\xi)-\hat u_n(z)|dz<\epsilon.	
		\end{equation}
		Combining inequalities \eqref{4.2.26}, \eqref{4.2.28} and \eqref{4.2.30},  we get
		\begin{equation*}
			\begin{split}
				&\iint_{\mathbb{R}^2}|\tilde u_n(x+\xi_1,{y}+\xi_2)-\tilde u_n(x,{y})|dxdy \\
				\leq& \iint_{\mathbb{R}^2}|\tilde u_n(x+\xi_1,{y}+\xi_2)-\hat u_n(x+\xi_1,{y}+\xi_2)|dxdy\\
				&+\iint_{\mathbb{R}^2}|\hat u_n(x+\xi_1,{y}+\xi_2)-\hat u_n(x,y)|dxdy
				+\iint_{\mathbb{R}^2}|\hat u_n(x,{y})-\tilde u_n(x,{y})|dxdy\\
				\leq& 2\|\tilde u_n-\hat u_n\|_{L^2(\mathbb{R}^2)}+\iint_{\mathbb{R}^2}|\hat u_n(x+\xi_1,{y}+\xi_2)-\hat u_n(x,{y})|dxdy\\
				\leq& 3\epsilon,
			\end{split}
		\end{equation*}
		and
		\begin{equation*}
			\begin{split}
				\iint_{\Omega-\bar G}|\tilde u_n(x, {y})|dxdy
				\leq& \iint_{\Omega-\bar G}|\tilde u_n(x,{y})-\hat u_n(x,{y})|dxdy\\
				&+\iint_{\Omega-\bar G} |\hat u_n(x,{y})|dxdy\leq 2\epsilon.
			\end{split}
		\end{equation*}
		So far, the inequalities  \eqref{4.2.25} are satisfied, so the sequence $\{\tilde u_n\}_{n\in\mathbb{N}}$ is precompact in $L^1(\Omega)$, and $\{u_n\}_{n\in\mathbb{N}}$ is the Cauchy sequence in $L^1(\Omega)$.
		
		{\it Step 2.} We claim that $$W_{0}^{1,1}(\Omega;x^\alpha)\hookrightarrow L^1(\Omega) ~\text{is compact}~.$$ 
		Indeed, by Lemma \ref{neicha4.2}, for any $4\ge q>2$, $\{u_n\}_{n\in\mathbb{N}}$ is bounded in $L^q(\Omega)$.  By interpolation inequality, we have 
		\begin{equation}\label{4.2.31}
			\|u_n-u_m\|_{L^2(\Omega)}\le \|u_n-u_m\|_{L^1(\Omega)}^\gamma \|u_n-u_m\|_{L^q(\Omega)}^{1-\gamma}\le (2C)^{1-\gamma}\|u_n-u_m\|_{L^1(\Omega)}^{\gamma}
		\end{equation}
		for some $\gamma \in (0,1)$. From Step 1, we know that $\{u_n\}_{n\in\mathbb{N}}$ is Cauchy sequence in $L^1(\Omega)$. So when $n,m$ is sufficiently large, combined with \eqref{4.2.31} we know that $\{u_n\}_{n\in\mathbb{N}}$ is the Cauchy sequence in $L^2(\Omega)$.
		Thus we have completed the proof.
	\end{proof}
	
	Next, we will prove Theorem~\ref{Intro-3}. Its proof needs the Kakutani fixed point theorem
	quoted from \cite{charalambos2013infinite}.
	\begin{lemma}\label{Intro-5}  Let $S$ be a nonempty, compact and convex subset of
		a locally convex Hausdorff space $X$. Let $\Phi: S\mapsto 2^S$ (where $2^S$ denotes the set consisting of
		all subsets of $S$)
		be a set-valued function
		satisfying:
		\begin{itemize}
			\item[(i)] For each $s\in S$, $\Phi(s)$ is a nonempty and convex subset;
			\item[(ii)] Graph $\Phi:=\{(s,z): s\in S\;\mbox{and}\;z\in \Phi(s)\}$ is closed.
		\end{itemize}
		Then the set of fixed points of $\Phi$ is nonempty and compact, where $s^*\in S$ is called to
		be a \emph{fixed point} of $\Phi$ if $s^*\in \Phi(s^*)$.
	\end{lemma}
	Now, we are position to prove Theorem ~\ref{Intro-3}.
	
	\begin{proof}[\textbf{Proof of Theorem~\ref{Intro-3}}]
		We first introduce  three set-valued functions
		$\Phi_1: {\mathcal{U}_1}\mapsto 2^{{\mathcal{U}_2}}, \Phi_2: {\mathcal{U}_2}\mapsto 2^{{\mathcal{U}_1}}$
		and $\Phi: {\mathcal{U}_1}\times {\mathcal{U}_2}\mapsto 2^{{\mathcal{U}_1}\times {\mathcal{U}_2}}$
		as follows:
		\begin{equation}\label{result-1}
			\Phi_1 f_1:= \{f_2\in {\mathcal U}_2: J_2(f_1,f_2)\leq J_2(f_1,v_2)\;\;
			\mbox{for all}\;\;v_2\in {\mathcal U}_2\},\;\;f_1\in {\mathcal U}_1,
		\end{equation}
		\begin{equation}\label{result-2}
			\Phi_2 f_2:= \{f_1\in {\mathcal U}_1: J_1(f_1,f_2)\leq J_1(v_1,f_2)\;\;
			\mbox{for all}\;\;v_1\in {\mathcal U}_1\},\;\;f_2\in {\mathcal U}_2,
		\end{equation}
		and
		\begin{equation}\label{result-3}
			\Phi(f_1,f_2):= \{({\widetilde f}_1,{\widetilde f}_2): {\widetilde f}_1\in \Phi_2 f_2\;\;\mbox{and}\;\;
			{\widetilde f}_2\in \Phi_1 f_1\},\;\;(f_1,f_2)\in {\mathcal U}_1\times {\mathcal U}_2.
		\end{equation}
		Then we set
		\begin{equation*}
			X:= (L^2(\Omega))^2\;\;\mbox{and}\;\;S:= {\mathcal U}_1\times {\mathcal U}_2.
		\end{equation*}
		It is clear that $X$ is a locally convex Hausdorff space.
		The rest of the proof will be carried out by the following four steps.
		
		{\it Step 1. Show that $S$ is a nonempty, compact and convex subset of $X$.}
		
		This fact can be easily checked. We omit the proofs here.
		
		{\it Step 2. $\Phi(f_1,f_2)$ is nonempty.} We prove that $\Phi(f_1,f_2)$ is nonempty for each $(f_1,f_2)\in S$.
		
		We arbitrarily fix $(f_1,f_2)\in S$. According to (\ref{result-1})-(\ref{result-3}), it suffices to
		show that $\Phi_1 f_1$ and $\Phi_2 f_2$ are nonempty. For this purpose, we introduce the following
		auxiliary optimal control problem:
		\begin{equation*}
			{\rm\bf (P_{au})}\;\;\;\;\;\;\;\inf_{v_2\in {\mathcal U}_2} J_2(f_1,v_2).
		\end{equation*}
		Let
		\begin{equation}\label{result-4}
			d:=\inf_{v_2\in {\mathcal U}_2} J_2(f_1,v_2).
		\end{equation}
		It is obvious that $d\geq 0$. Let $\{v_{2,n}\}_{n\geq 1}\subseteq {\mathcal U}_2$ be a minimizing sequence
		so that
		\begin{equation}\label{result-5}
			d=\lim_{n\rightarrow \infty} J_2(f_1,v_{2,n}).
		\end{equation}
		On one hand, since $\|v_{2,n}\|_{L^2(\Omega;x^{-\alpha})}\leq M_2$,
		there exists a subsequence of $\{v_{2,n}\}_{n\geq 1}$, still denoted by itself, and
		$v_{2,0}\in {\mathcal U}_2$, so that
		\begin{equation}\label{result-5:1}
			v_{2,n}\rightarrow v_{2,0}\;\;\mbox{weakly in}\;\;L^2(\Omega;x^{-\alpha}).
		\end{equation}
		
		We denote that $\bar{u}_n:=u(g,f_1,v_{2,n})-u(g,f_1,v_{2,0})$. According to (\ref{1.1}), it is clear that
		\begin{equation}
			\label{result-5:2}
			\begin{cases}
				-\frac{1}{2}\partial_{xx}\bar{u}_n+x^{\frac{\alpha}{2}}\partial_{y}\bar{u}_n=\chi_{\omega_2}(v_{2,n}-v_{2,0}), & \mbox{in}\,\, \Omega,\\
				\bar{u}_n=0, & \mbox{on}\,\, \partial \Omega.
			\end{cases}
		\end{equation}
		By \eqref{ine44}, we have 
		$$
		\|\bar{u}_n\|_{W^{1,1}(\Omega;x^{\alpha})}\leq C\|\chi_{\omega_2}(v_{2,n}-v_{2,0})\|_{L^2(\Omega;x^{-\alpha})},
		$$
		which along with (\ref{result-5:1})
		, implies that there exists a subsequence of $\bar{u}_n$, still denoted by itself, and $\bar{u}\in W^{1,1}(\Omega;x^{\alpha})$ such that
		\begin{align}
			\label{result-5:7}
			\bar{u}_n\rightarrow \bar{u}\;\;\mbox{weakly in}\;\;W^{1,1}(\Omega;x^{\alpha}).
		\end{align}
		By Proposition \ref{compactembed}, the embedding $W^{1,1}(\Omega;x^{\alpha})\hookrightarrow L^2(\Omega)$ is compact, which implies that
		\begin{align}
			\label{result-5:73}
			\bar{u}_n\rightarrow \bar{u}\;\;\mbox{strongly in}\;\;L^2(\Omega).
		\end{align}
		Passing to the limit for $n\rightarrow\infty$ in (\ref{result-5:2}), by (\ref{result-5:1}) and (\ref{result-5:7}), we have
		\begin{equation*}
			\begin{cases}
				-\frac{1}{2}\partial_{xx}\bar{u}+x^{\alpha}\partial_{y}\bar{u}=0, & \mbox{in}\,\, \Omega,\\
				\bar{u}_n=0, & \mbox{on}\,\, \partial \Omega,
			\end{cases}
		\end{equation*}
		from which we obtain that $\bar{u} = 0$. Hence, we have
		\begin{equation}\label{result-6}
			u(g,f_1,v_{2,n})\rightarrow u(g,f_1,v_{2,0}) ~\text{strongly in}~ L^2(\Omega).
		\end{equation}
		
		On the other hand, from (\ref{result-5:1}), we have
		\begin{equation}\label{9}
			\|v_{2,0}\|_{L^2(\Omega;x^{-\alpha})}\leq\liminf_{n\rightarrow\infty} \|v_{2,n}\|_{L^2(\Omega;x^{-\alpha})}.
		\end{equation}
		Now for any $\varphi\in C_0^\infty(\Omega)$, multiplying the first equation of the following equation by $\varphi$ and integrating over $\Omega$:
		\begin{equation*}
			\begin{cases}
				-\frac{1}{2}\partial_{xx}u+x^{\alpha}\partial_{y}u=\chi_{\omega}g+\chi_{\omega_1}f_1+\chi_{\omega_2}v_{2,n}, & \mbox{in}\,\, \Omega,\\
				u=0, & \mbox{on}\,\, \partial \Omega,
			\end{cases}
		\end{equation*}
		and by (\ref{result-5:1}), it implies that
		\begin{equation*}
			\begin{cases}
				-\frac{1}{2}\partial_{xx}u+x^{\alpha}\partial_{y}u=\chi_{\omega}g+\chi_{\omega_1}f_1+\chi_{\omega_2}v_{2,0}, & \mbox{in}\,\, \Omega,\\
				u=0, & \mbox{on}\,\, \partial \Omega.
			\end{cases}
		\end{equation*}
		Therefore, it follows from (\ref{result-5}), (\ref{Intro-2}),  (\ref{result-6}) and (\ref{9}) that
		\begin{equation}\label{result-7}
			d=J_2(f_1,v_{2,0}).
		\end{equation}
		Noting that $v_{2,0}\in {\mathcal U}_2$, by (\ref{result-4}), (\ref{result-7}) and (\ref{result-1}),
		we obtain that $v_{2,0}\in \Phi_1 f_1$. This implies that $\Phi_1 f_1\not=\emptyset$. In the same way,
		we also have that $\Phi_2 f_2\not=\emptyset$.
		
		{\it Step 3. Convex subset $\Phi(f_1,f_2)$.} We show that $\Phi(f_1,f_2)$ is a convex subset of ${\mathcal U}_1\times {\mathcal U}_2$ for
		each $(f_1,f_2)\in {\mathcal U}_1\times {\mathcal U}_2$.
		
		We arbitrarily fix $(f_1,f_2)\in {\mathcal U}_1\times {\mathcal U}_2$.
		According to (\ref{result-1})-(\ref{result-3}), it suffices to prove that
		$\Phi_1 f_1$ is a  convex subset of ${\mathcal U}_2$. The convexity of $\Phi_2 f_2$ can
		be similarly proved. For this purpose, we arbitrarily fix ${\widetilde f}_2, {\widehat f}_2\in \Phi_1 f_1$.
		By (\ref{result-1}),
		we get that
		\begin{equation}\label{result-8}
			{\widetilde f}_2,\;\;{\widehat f}_2\in {\mathcal U}_2,
		\end{equation}
		and
		\begin{equation}\label{result-9}
			J_2(f_1,{\widetilde f}_2)\leq J_2(f_1,v_2)\;\;\mbox{and}\;\;
			J_2(f_1,{\widehat f}_2)\leq J_2(f_1,v_2)\;\;\mbox{for each}\;\;v_2\in {\mathcal U}_2.
		\end{equation}
		For any $\lambda\in [0,1]$, by (\ref{Intro-2}), we have that
		\begin{eqnarray*}
			&&J_2(f_1,\lambda {\widetilde f}_2+(1-\lambda){\widehat f}_2)-[\lambda J_2(f_1,{\widetilde f}_2)+(1-\lambda)J_2(f_1,{\widehat f}_2)]\\[3mm]
			&\leq&\|u(g,f_1,\lambda {\widetilde f}_2+(1-\lambda){\widehat f}_2)-y^2_d\|^2_{L^2(G_2)}-\lambda\|u(g,f_1,{\widetilde f}_2)-y^2_d\|^2_{L^2(G_2)}\\[3mm]
			&&-(1-\lambda)\|u(T;y_0,u_1,{\widehat u}_2)-y^2_T\|^2_{L^2(G_2)}\\[3mm]
			&=&\|\lambda [u(g,f_1,{\widetilde f}_2)-y^2_d]+(1-\lambda)[u(g,f_1,{\widehat f}_2)-y^2_d]\|^2_{L^2(G_2)}\\[3mm]
			&&-\lambda\|u(g,f_1,{\widetilde f}_2)-y^2_d\|^2_{L^2(G_2)}-(1-\lambda)\|u(g,f_1,{\widehat f}_2)-y^2_d\|^2_{L^2(G_2)}\\[3mm]
			&\leq&0.
		\end{eqnarray*}
		This, along with (\ref{result-8}) and (\ref{result-9}), yields that
		\begin{equation*}
			\lambda {\widetilde f}_2+(1-\lambda){\widehat f}_2\in {\mathcal U}_2
		\end{equation*}
		and
		\begin{equation*}
			J_2(f_1,\lambda{\widetilde f}_2+(1-\lambda){\widehat f}_2)\leq J_2(f_1,v_2)\;\;\mbox{for each}\;\;v_2\in {\mathcal U}_2,
		\end{equation*}
		which indicate that $\lambda{\widetilde f}_2+(1-\lambda){\widehat f}_2\in \Phi_1 f_1$ (see (\ref{result-1})). Hence, $\Phi_1 f_1$
		is a convex subset of ${\mathcal U}_2$.
		
		{\it Step 4. Prove that Graph $\Phi$ is closed.}
		
		It suffices to show that if $(f_{n,1},f_{n,2})\in {\mathcal U}_1\times {\mathcal U}_2,
		{\widetilde f}_{n,1}\in \Phi_2 f_{n,2}$, ${\widetilde f}_{n,2}\in \Phi_1 f_{n,1}$,
		$(f_{n,1},f_{n,2})\rightarrow (f_1,f_2)$ in $X$
		and $({\widetilde f}_{n,1},{\widetilde f}_{n,2})\rightarrow ({\widetilde f}_1,{\widetilde f}_2)$  in
		$X$, then
		\begin{equation}\label{20180208-3}
			(f_1,f_2)\in {\mathcal U}_1\times {\mathcal U}_2,\; {\widetilde f}_1\in \Phi_2 f_2\;\;
			\mbox{and}\;\;{\widetilde f}_2\in \Phi_1 f_1.
		\end{equation}
		Indeed, on one hand, by (\ref{result-1}) and (\ref{result-2}), we can easily check that
		\begin{equation}\label{result-10}
			(f_1,f_2)\in {\mathcal U}_1\times {\mathcal U}_2,\; {\widetilde f}_1\in {\mathcal U}_1\;\;
			\mbox{and}\;\;{\widetilde f}_2\in {\mathcal U}_2.
		\end{equation}
		On the other hand, according to ${\widetilde f}_{n,1}\in \Phi_2 f_{n,2}$, (\ref{result-2}) and (\ref{Intro-2}), it is obvious that for each $v_1\in {\mathcal U}_1$,
		\begin{equation}\label{result-11}
			\|u(g,{\widetilde f}_{n,1},f_{n,2})-y^1_d\|_{L^2(G_1)}\leq \|u(g,v_1,f_{n,2})-y^1_d\|_{L^2(G_1)}.
		\end{equation}
		Since $({\widetilde f}_{n,1},f_{n,2})\rightarrow ({\widetilde f}_1,f_2)$ weakly in $(L^2(\Omega;x^{-\alpha}))^2$, by similar arguments as
		those to get (\ref{result-6}), there exists a subsequence of $\{n\}_{n\geq1}$, still denoted by itself, so that
		\begin{equation*}
			\begin{array}{lll}
				&(u(g,{\widetilde f}_{n,1},f_{n,2}),u(g,v_1,f_{n,2}))\\[3mm]
				& \rightarrow (u(g,{\widetilde f}_1,f_2),u(g,v_1,f_2))
				~\text{strongly in}~  (L^2(\Omega))^2,
			\end{array}
		\end{equation*}
		which, implies that
		\begin{equation}\label{result-12}
			\begin{split}
				&(u(g,{\widetilde f}_{n,1},f_{n,2}),u(g,v_1,f_{n,2}))\\[3mm]
				& \rightarrow (u(g,{\widetilde f}_1,f_2),u(g,v_1,f_2))
				~\text{strongly in}~  (L^2(G_1))^2.
			\end{split}
		\end{equation}
		Passing to the limit for $n\rightarrow \infty$ in (\ref{result-11}), by (\ref{result-12}),
		we get that for each $v_1\in {\mathcal U}_1$,
		\begin{equation*}
			\|u(g,{\widetilde f}_1,f_2)-y_d^1\|_{L^2(G_1)}\leq \|u(g,v_1,f_2)-y_d^1\|_{L^2(G_1)}.
		\end{equation*}
		This, together with (\ref{Intro-2}), (\ref{result-2}) and the second conclusion in (\ref{result-10}), implies that
		${\widetilde f}_1\in \Phi_2 f_2$. Similarly, ${\widetilde f}_2\in \Phi_1 f_1$. Hence, (\ref{20180208-3}) follows.
		
		{\it Step 5. Finish the proof.}
		
		According to Steps 1-4 and Lemma~\ref{Intro-5}, there exists a pair of $(f_1^*,f_2^*)\in {\mathcal U}_1\times {\mathcal U}_2$
		so that $(f_1^*,f_2^*)\in \Phi(f_1^*,f_2^*)$,
		which, combined with (\ref{result-1})-(\ref{result-3}), indicates that
		$(f_1^*,f_2^*)$ is a Stackelberg-Nash equilibrium of the problem {\bf(P)}.
		
		In summary, we end the proof of Theorem~\ref{Intro-3}.
	\end{proof}

	\section{Acknowledgement}
	
	This work is supported by the National Natural Science Foundation of China, the Science-Technology Foundation of Hunan Province.

	\bibliographystyle{abbrvnat}
	\bibliography{ref.bib}
	%
	%
	%
\end{document}